\documentclass[11pt]{article}
\usepackage{amsmath,amsthm,amssymb}
\usepackage{mathtools}
\usepackage[T1]{fontenc}
\usepackage[utf8]{inputenc}
\usepackage[dvipdf]{graphicx}
\usepackage{color}
\usepackage{epstopdf}
\usepackage{dsfont}
\usepackage{enumerate}
\usepackage{enumitem}
\usepackage{bbm}

\definecolor{db}{RGB}{0, 0, 130}
\usepackage[colorlinks=true,citecolor=red,linkcolor=db,urlcolor=blue,pdfstartview=FitH]{hyperref}

\usepackage[normalem]{ulem}
\usepackage[numbers]{natbib}

\definecolor{rp}{rgb}{0.25, 0, 0.75}
\definecolor{dg}{rgb}{0, 0.6, 0}

\newtheorem{theorem}{Theorem}[section]

\newtheorem{definition}{Definition}[section]

\newtheorem{assumption}[theorem]{Assumption}
\newtheorem{lemma}[definition]{Lemma}

\newtheorem{proposition}[definition]{Proposition}
\newtheorem{remark}[definition]{Remark}

\def\1{\mathbbm{1}}
\def\K{\mathbb{K}}
\def\R{\mathbb{R}}

\def\E{\mathbb{E}}
\def\N{\mathbb{N}}

\def\x{\times}
\def\Om{\Omega}

\def\Fc{\mathcal{F}}
\def\F{\mathbb{F}}
\def\P{\mathbb{P}}

\def\gammab{\bar{\gamma}}

\def\eps{\varepsilon}

\def\Gc{\mathcal{G}}
\def\Lc{\mathcal{L}}
\def\Pc{\mathcal{P}}

\def\Wc{\mathcal{W}}
\def\Mc{\mathcal{M}}

\def\Xk{X^k}

\def\Wk{W^k}
\def\Qk{Q^k}
\def\<{\langle}
\def\>{\rangle}

\DeclareMathOperator{\Tr}{Tr}
\def\gammab{\bar \gamma}

\def\Ac{\mathcal{A}}

\def\ie{\textit{i.e.}}

\title{Comparison of viscosity solutions for a class of non-linear PDEs on the space of finite nonnegative measures}

\author{
	Ibrahim Ekren \thanks{Department of Mathematics, University of Michigan. iekren@umich.edu. I. Ekren is partially supported by
	the NSF grant DMS-2406240}
	\and Xihao He\thanks{Department of Mathematics, University of Michigan. hexihao@umich.edu}
	\and Tianxu Lan\thanks{Department of Mathematics, The Chinese University of Hong Kong. 1155184513@link.cuhk.edu.hk}
        \and Xiaolu Tan\thanks{Department of Mathematics, The Chinese University of Hong Kong. xiaolu.tan@cuhk.edu.hk. X. Tan is supported
by Hong Kong RGC General Research Fund (projects 14302921 and 14302622)}
}

\date{\today}

\begin{document}

\maketitle

\begin{abstract}
We establish a comparison principle for viscosity solutions of a class of nonlinear partial differential equations posed on the space of nonnegative finite measures, thereby extending recent results for PDEs defined on the Wasserstein space of probability measures.
As an application, we study a controlled branching McKean–Vlasov diffusion and characterize the associated value function as the unique viscosity solution of the corresponding Hamilton–Jacobi–Bellman equation. This yields a PDE-based approach to the optimal control of branching processes.
	\end{abstract}
	

\tableofcontents

\section{Introduction}

The McKean--Vlasov dynamic is motivated to study the limiting behaviour of a large population particle system with mean-field interaction, as the population size tends to infinity, see e.g. \cite{Dawson Vaillancourt 1995, Sznitman 1991}, etc.
Recently, optimal control problem of the McKean-Vlasov dynamics has been extensively studied in the literature, see e.g. \cite{Carmona Delarue 2013, Pham Wei 2018}, which leads to HJB equations on the Wasserstein space, see e.g. \cite{Pierre Cardaliaguet 2019,ReneCarmona2018,Alekos Cecchin 2024,Wilfrid Gangbo 2022,Erhan 2018,Andrea Cosso 2024}. PDEs in the space of measures also appear in mean-field optimal stopping problems \cite{Talbi 2023a, Talbi 2023b, Possamai 2023}, mean-field control of jump processes \cite{Burzoni 2020}, and control problems of occupied processes \cite{Soner 2024a}.
A prominent technique for analyzing the associated HJB equations is the "lifting" method, which transforms the infinite-dimensional PDE on the Wasserstein space into a finite-dimensional PDE on a Hilbert space of random variables. For instance, in \cite{Jianjun Zhou 2024}, the authors introduced a notion of viscosity solutions by lifting the equation to the space of processes, proving existence and uniqueness under Lipschitz continuity. Alternatively, a direct approach working on the Wasserstein space itself has been developed. The idea of using the Fourier-Wasserstein distance in the doubling variable argument was brought up in \cite{Soner 2024b, Soner 2024c} to show the comparison principle for first-order PDEs in the Wasserstein space. It was then adopted in \cite{Erhan Xin 2025, Daudin 2025} for the comparison of second-order PDEs on space of probability measures.
Another innovative technique, employing a particle approximation argument, was used in \cite{Erhan Hang 2025} to prove a comparison principle for a specific HJB equation under weaker assumptions.

\vspace{0.5em}

In this paper, we study a class of second-order PDEs on the space $\Mc_2(\R^d)$ of measures with finite 2nd order moment, rather than the space of probability measures in the classical literature, in the form:
\begin{equation}\label{eq:second-order}
	\begin{cases}
		-\left(\partial_t v + G\left( \cdot, \partial_{x}{\delta_{\mu} v}, \partial^2_{xx}{\delta_{\mu} v}, {\delta_{\mu} v} \right) \right)(t,m)
		~=~
		0, & \text{on } [0,T)\times\Mc_2(\R^d),\\
		v(T,m)
		~=~
		\<g(\cdot, m), m\>,  & \text{on } \Mc_2(\R^d).
	\end{cases}
\end{equation}
Our main objective is to establish the comparison principle for semi-continuous viscosity solutions under some technical conditions (c.f. Assumptions \ref{A.ham}).
Further, as one of the applications, we study an optimal control problem of McKean–Vlasov branching diffusions, in which each particle evolves according to a stochastic differential equation whose coefficients depend on a time‐varying mean measure, and may branch at a controlled rate.

\vspace{0.5em}

The optimal control of stochastic systems with branching mechanisms has attracted considerable interest due to its applications in biology, finance, and multi-agent systems. Branching diffusion processes extend classical diffusion models by allowing particles to reproduce or disappear randomly, leading to a population whose size and distribution evolve stochastically. 
When branching is introduced, the problem becomes more complex due to the random fluctuation in population size and the need to track both the number and spatial distribution of particles. 
In the setting without mean-field interaction, the control problem of branching process has been studied in \cite{Nisio 1985}, and recently in \cite{Claisse 2018, KharroubiOcello,Ocello 2023}, etc. 
More recently, \cite{Claisse Kang Tan 2024} studies a general McKean--Vlasov SDE with branching, and provides a well-posedness as well as a propagation of chaos result. A quantitative weak propagation of chaos result is then obtained in \cite{CaoRenTan}. In a related context, the mean-field games with branching has been studied in \cite{Claisse Ren Tan 2019}.
In \cite{Claisse Kang Lan Tan 2025}, an optimal control problem of the branching McKean--Vlasov SDE has been studied, with the corresponding HJB master equation defined on the space of finite measures, but based on the notion of the classical solution.

\vspace{0.5em}

The main contribution in this work is twofold. First, we provide a comparison principle for semicontinuous viscosity solutions for a class of non-linear PDEs on $\Mc_2(\R^d)$, which generalizes the results of \cite{Erhan Xin 2025, Erhan He 2025} from the space of probability measures to the space of finite measures. Beyond branching processes, working on the space of nonnegative finite measures, rather than probability measures, arises naturally in situations where the total mass of the system is not conserved, such as control problems involving absorption, killing, or stopping; see, for instance, \cite{jackson,stefan}.
A key technical novelty here is the design of a new auxiliary function \eqref{aux} to establish compactness, crucially containing a term $m^2(\R^d)$ to penalize sequences with unbounded mass, which is necessary for the doubling variable argument on $\Mc_2(\R^d)$. 
Secondly, we provide a first PDE characterization of the value function for a controlled McKean-Vlasov branching diffusion problem, identifying it as the unique viscosity solution of the corresponding HJB equation. This provides a PDE-based approach to the optimal control problem of the branching McKean-Vlasov dynamics.

\vspace{0.5em}

To establish the viscosity solution theory for \eqref{eq:second-order}, the lifting technique \cite{Jianjun Zhou 2024} fails because the randomly fluctuating population size in a branching process prevents a fixed lifting to a process space. 
On the other hand, the particle approximation method \cite{Erhan Hang 2025} relies on the limit theory of the McKean-Vlasov control problem, which seems natural but would not be direct and trivial to be extended for the controlled McKean-Vlasov branching processes problem, see also the propagation of chaos result for uncontrolled McKean-Vlasov branching diffusion in \cite{Claisse Kang Tan 2024, CaoRenTan}.
To overcome these limitations, we adopt an intrinsic methodology inspired by \cite{Erhan He 2025}, and work directly on the space of finite measures $\Mc_2(\R^d)$ and employ a notion of viscosity solutions based on the linear functional derivative. This approach naturally accommodates the varying total mass of the branching marginal measures.
Under assumptions analogous to those in \cite{Claisse Kang Lan Tan 2025}, we first establish the existence of a viscosity solution to the corresponding HJB equation \eqref{eq:second-order} showing that the value function of the control problem is indeed a viscosity solution. In particular, this extends the methodology of \cite{Andrea Cosso 2024}—originally developed for classical McKean-Vlasov control—to the branching setting, while relaxing certain technical conditions. 
Finally, employing techniques from \cite{Erhan He 2025}, we verify that the coefficients of the HJB equation satisfy Assumption \ref{A.ham}, thereby allowing the uniqueness result from our earlier PDE analysis to apply. This leads to the characterization of the value function as the unique viscosity solution of an equation of the form \eqref{eq:second-order}.

\vspace{0.5em}

The remainder of the paper is organized as follows. In Section \ref{section2}, we present notations and
definitions of viscosity solutions. In Section \ref{section3}, we prove a general comparison principle for semicontinuous
solutions. In Section \ref{section4}, we introduce the branching dynamic, the closed-loop control problem and the associated value function, as well as our main application result. In Section \ref{section5}, we provide a priori estimate and use it to prove the rest of the results in Section \ref{section5}.

\section{Preliminaries}\label{section2}

\subsection{Notations}
	Let us first introduce some notations used in the rest of the paper.

	\vspace{0.5em}

	\noindent $\mathrm{(i)}$ Let $(X, \rho)$ be a nonempty metric space, we denote by $\Pc(X)$ (resp. $\Mc(X)$) the space of all Borel probability measures (resp. nonnegative finite measures) on $X$.  
	Given a measure $\mu\in\Mc(X)$ and a measurable mapping $f:X\to\R,$ we denote the integration by
	\begin{equation*}
	 \langle\mu,f\rangle ~:=~ \int_E f(x) \,\mu(dx).
	\end{equation*}
   Further, for $p\geq1$, we denote by $\Pc_p(X)$ the space of probability measures on $X$ with $p$-order moment, \ie,
		$$
			\Pc_p(X)~:=~\left\{\mu\in \Pc(X):\int_{X}\rho(x,x_0)^p \mu(d x)<+\infty\right\},
		$$
		for some (and thus all) fixed point $x_0\in X$.
		Let us equip $\Pc_p(X)$ with the Wasserstein metric $\Wc_p$ given by
		$$
			\Wc_p(\mu,\nu) ~:=~ \inf_{\lambda \in\Lambda(\mu,\nu)} \left(\int_{X\x X} \rho(x, y)^p\lambda(dx,dy) \right)^{1/p},
		$$
		where $\Lambda(\mu,\nu)$ is the collection of all Borel probability measures $\lambda$ on $X\x X$ whose marginals are $\mu$ and $\nu$ respectively.
		Similarly, we denote by $\Mc_p(X)$ the collection of all finite non-negative Borel measures on $X$ with $p$-order moment, \ie,
		$$
			\Mc_p(X)~:=~\left\{\mu\in \Mc(X):\int_{X}\rho(x,x_0)^p \mu(d x)<+\infty\right\}.
		$$
		In particular, when $X = \R^d$, 
		$$
			\Mc_p(\R^d) ~:=~ \left\{\mu\in \Mc(\R^d):\int_{\R^d} |x|^p \mu(d x)<+\infty\right\}.
		$$

		Next, let us denote by $\mathbb{S}_{d}$ the space of all $d \times d$-dimensional matrices with real entries, equipped with the Frobenius Norm $|\cdot|$. 
		Let us denote by $B_{l}^{d}$ ($B_{l}^{d \times d}$) the set of Borel measurable functions on $\mathbb{R}^{d}$ ($\mathbb{R}^{d \times d}$) with at most linear growth, and define $|f|_{l} = \sup_{x \in \mathbb{R}^{d}} \frac{|f(x)|}{1+|x|}$ for any $f \in B_{l}^{d}$ ($|g|_{l} := \sup_{x \in \mathbb{R}^{d \times d}} \frac{|g(x)|}{1+|x|}$ for any $g \in B_{l}^{d \times d}$). We also denote the space $B_{l}^{d} \cap H^\lambda$ ($B_{l}^{d \times d} \cap H^\lambda$) by $B_{l}^{d,\lambda}$ ($B_{l}^{d \times d,\lambda}$), where $H^\lambda$ is the Soblev spaces as defined in Section \ref{section2.3}.

		\vspace{0.5em}

		\noindent $\mathrm{(ii)}$ To describe the progeny of the branching process, we use the classical Ulam--Harris--Neveu notation.
		Let
		$$
			\K:= \{\emptyset\}\cup\bigcup_{n=1}^{+\infty}\N^n.
		$$
		Given $k, k'\in\K$ with $k=k_1...k_n$ and $k'=k'_1...k'_m$, we define the concatenation of labels $kk':= k_1...k_nk'_1...k'_m$, and denote by $k\prec k'$ if there exists $\tilde{k}$ such that $k'=k\tilde{k}$. 
		Let us define
		$$
			E 
			~:=~
			\left\{
			\sum_{k\in K} \delta_{(k,x^k)} :
			K \subset \K ~\text{finite},
			x^k\in \R^d,
			\text{for all}~ k, k' \in K, k\nprec k'
			\right\}.
		$$
		Let $\K$ be equipped with the discrete topology, then $E$ is a closed subspace of $\Mc(\K\x\R^d)$ under the weak convergence topology and thus $E$ is a Polish space.
		We provide a metric $d_E$ on $E$ which is consistent with the weak convergence topology:
		for all $e_1, e_2\in E$ such that $e_1 = \sum_{k\in K_1} \delta_{(k,x^k)}$ and $e_2 = \sum_{k\in K_2} \delta_{(k,y^k)}$,
		\begin{equation}\label{metric:E}
			d_E(e_1,e_2)
			~:=
			\sum_{k\in K_1\cap K_2}                                                                                                                   
			\left(|x^k-y^k|\wedge1\right)
			~+~
			\#(K_1\triangle K_2),
		\end{equation}
		where $K_1\triangle K_2 := (K_1\setminus K_2) \cup (K_2\setminus K_1)$ and $\#(K_1\triangle K_2)$ is the number of element in $K_1\triangle K_2$.
		Let $f:=(f^k)_{k\in\K}:\K\x\R^d\rightarrow \R$ be a function and $e:=\sum_{k\in K}\delta_{(k,x^k)}\in E$, we have
		\begin{align*}
			\<e,f\>
			~=~
			\sum_{k\in K}f^k(x^k).
		\end{align*}
		We also fix the reference point $e_0 \in E,$ the null measure,  which means that the associated set of particle is empty. 

\subsection{Some metrics}

In the analysis of partial differential equations on measure spaces, the choice of metric is crucial, as it dictates the topology, the notion of convergence, and the properties of the functionals involved.
We now aim to define two metrics on $\mathcal{M}_2(\mathbb{R}^{d})$.
\subsubsection{Fourier Wasserstein metric}
		First we shall introduce a Fourier-Wasserstein metric on $\mathcal{M}(\mathbb{R}^{d})$. Let
		\begin{equation}\label{eq:lambda_def}
			\lambda := 
			\begin{cases} 
			\left\lfloor\frac{d}{2}\right\rfloor + 4, & \text{for } d = 4k, 4k + 1, \\
			\left\lfloor\frac{d}{2}\right\rfloor + 3, & \text{for } d = 4k + 2, 4k + 3. 
			\end{cases}
			\end{equation}
The Fourier Wasserstein distance $\rho_{F}$ between two finite measures $m_{1}$ and $m_{2}$ with $m_{1},m_{2}\in\mathcal{M}(\mathbb{R}^{d})$ is defined as

\begin{align}\label{f-metric}
	\rho_{F}^{2}(m_{1},m_{2}) := \int_{\mathbb{R}^{d}} \frac{|F_{n}(m_{1} - m_{2})|^{2}}{(1 + |n|^{2})^{\lambda}}  dn,
\end{align}
where for $n \in \mathbb{R}^{d}$, $f_{n}(x) := (2\pi)^{-\frac{d}{2}} e^{in \cdot x}$, and $F_{n}(\rho) := \langle f_{n}, \rho \rangle$. For $\theta_{1}=(t_{1},m_{1}), \theta_{2}=(t_{2},m_{2})\in\Theta:=[0,T]\times\mathcal{M}(\mathbb{R}^{d})$, define
\[
d_{F}^{2}(\theta_{1},\theta_{2})=|t_{1}-t_{2}|^{2} + \rho_{F}^{2}(m_{1},m_{2}).
\]

\subsubsection{Bounded Lipschitz distance}\label{T_metric}

We also recall the bounded Lipschitz distance $\mathbf{d}_{\mathrm{BL}}$ on $\Mc(\R^d)$:
	$$
		\mathbf{d}_{\mathrm{BL}} \big(\mu_1, \mu_2 \big)
		~=~
		\sup_{\varphi\in \mathrm{Lip}_1(\R^d)}\big\{\<\mu_1,\varphi\> - \<\mu_2,\varphi\>\big\},
	$$
	where
	\begin{align*}
		\text{Lip}_1(\R^d)
		~:=~
		\bigg\{
			\varphi:\R^d\rightarrow\R~:~
			~\big|\varphi(x) - \varphi(y)\big| \leq |x-y|,~\text{and}~\big|\varphi(x)\big| \le 1,
			~\mbox{for all}~ x,y\in\R^d
		\bigg\}.
	\end{align*}
	
	Notice that the bounded Lipschitz distance $\mathbf{d}_{\mathrm{BL}}$ is equivalent to the extended Wasserstein metric $\bar \Wc_1$ in \cite[Appendix~B]{Claisse Ren Tan 2019} when $\R^d$ is equipped with the truncated Euclidean norm $| \cdot | \wedge 1$, and they both metrize the weak convergence topology on the space $\mathcal{M}(\mathbb{R}^{d})$, see \cite[Remark 2.2]{CaoRenTan} for details.

	\vspace{1em}

	Understanding the relationship between the above two metrics and the topological structures they induce is fundamental for the subsequent analysis, particularly for the comparison principle and its application to the control problem. The following lemma establishes a direct inequality between $\rho_F$ and $\mathbf{d}_{\mathrm{BL}}$, which is essential for transferring estimates and properties from one metric framework to the other.

	\begin{lemma}
		For any \(m,m' \in \Mc(\R^d)\),
		$$\rho_F(m, m') \le C \mathbf{d}_{\mathrm{BL}}(m, m').$$
	\end{lemma}

	\begin{proof}
		Let $\eta = m - m'$. For each $n \in \mathbb{R}^{d}$, define the function
	\[
	\psi_{n}(x) = \frac{e^{i n \cdot x} - e^{i n \cdot x_{0}}}{L_{n}},
	\]
	where $x_{0} \in \mathbb{R}^{d}$ is the fixed point, and $L_{n} = \max(|n|, 2)$. It is easy to verify that $\psi_{n} \in \operatorname{Lip}_{1}(\mathbb{R}^{d})$.
	From the definition of $\mathbf{d}_{\mathrm{BL}}$,
	$$\left| \int_{\mathbb{R}^d} \psi_n(x) \eta(dx) \right| \le \mathbf{d}_{\mathrm{BL}}(m, m').$$
	Substituting the expression of \(\psi_{n}\):
	\[
	\left|(2\pi)^{d/2}F_n(\eta)- e^{i n \cdot x_{0}} \int d\eta(x)\right|=\left|\int e^{i n \cdot x} \, d\eta(x) - e^{i n \cdot x_{0}} \int d\eta(x)\right| \leq L_{n} \mathbf{d}_{\mathrm{BL}}(m, m'),
	\]
and	
	\[
	\left|(2\pi)^{d/2}F_n(\eta) - e^{i n \cdot x_{0}} \eta(\mathbb{R}^{d})\right| \leq L_{n} \mathbf{d}_{\mathrm{BL}}(m, m').
	\]
	By the triangle inequality:
\begin{align*}
(2\pi)^{d/2}|F_n(\eta)| &\leq \left|(2\pi)^{d/2}F_n(\eta)- e^{i n \cdot x_{0}} \eta(\mathbb{R}^{d})\right| + \left|e^{i n \cdot x_{0}} \eta(\mathbb{R}^{d})\right| \\
&\leq L_{n} \mathbf{d}_{\mathrm{BL}}(m, m') + |\eta(\mathbb{R}^{d})| \leq (L_{n} + 1) \mathbf{d}_{\mathrm{BL}}(m, m').
\end{align*}	
	Since $L_{n} \leq |n| + 2$, we have $L_{n} + 1 \leq |n| + 3$, therefore
	\[
	|F_n(\eta)| \leq (|n| + 3) (2\pi)^{d/2}\mathbf{d}_{\mathrm{BL}}(m, m').
	\]
	Hence,
	\[
	\rho_{F}^{2}\left(m, m^{\prime}\right)=\int_{\mathbb{R}^{d}}\frac{\left|F_{n}(\eta)\right|^{2}}{\left(1+|n|^{2}\right)^{\lambda}}  dn\leq(2\pi)^{-d} \bigg( \int_{\mathbb{R}^{d}}\frac{(|n|+3)^{2}}{\left(1+|n|^{2}\right)^{\lambda}}  dn \bigg)~ \mathbf{d}_{\mathrm{BL}}^2(m, m').
	\]
	\end{proof}

\subsection{Sobolev norm}\label{section2.3}
This section briefly recalls the definition of Sobolev spaces via Bessel potentials and a key multiplicative theorem. These concepts are essential for the analysis in Section \ref{section5.2} and Section \ref{section5.3}.

The class of Schwartz functions $\mathcal{S}(\mathbb{R}^{d})$ is the space of smooth functions whose derivatives are bounded by $C_{N}(1+|\xi|^{2})^{-N}$ for every $N\in\mathbb{Z}^{+}$, equipped with the topology induced by the family of seminorms

\[
\rho_{\alpha,\beta}(\phi)=\sup_{\xi\in\mathbb{R}^{d}}|\xi^{\alpha}\partial_{\beta}\phi(\xi)|,\quad\forall\,\phi\in\mathcal{S}(\mathbb{R}^{d}),
\]
indexed by all multi-indices $\alpha,\beta$. Denote by $\mathcal{S}^{\prime}(\mathbb{R}^{d})$ the topological dual of $\mathcal{S}(\mathbb{R}^{d})$.

\begin{definition}\label{Bessel operator}
	Let $s$ be a real number. The space $H^{s}(\mathbb{R}^{d})$ is defined as the set of all finite measure $u$ in $\mathcal{S}^{\prime}(\mathbb{R}^{d})$ with the property that
\[
\mathcal{J}_{-s}(u):=(I_{d}-\Delta)^{s/2}u=\mathcal{F}^{-1}\left((1+|\xi|^{2})^{s/2}\mathcal{F}u(\cdot)\right) 
\]
is an element of $L^{2}(\mathbb{R}^{d})$. Here, $\mathcal{F}$ denotes the Fourier transform. $\mathcal{J}_{-s}$ is called the Bessel potential operator, and $|u|_{s}:=|\mathcal{J}_{-s}u|_{L^{2}}=|(1+|\xi|^{2})^{s/2}\mathcal{F}u(\cdot)|_{L^{2}}$.
\end{definition}
We recall the following theorem, the proof of which can be found in {\cite[Proposition 5.6]{Ali Behzadan 2021}}

\begin{theorem}\label{soblev multiplication}
	Assume that $s_{i} \in \mathbb{R}$ ($i=1,2$) and $s>0$ are real numbers that satisfy $s_{i} \geq s$, $s_{1}+s_{2}-s > \frac{d}{2}$. Then the pointwise multiplication of functions extends uniquely to a continuous bilinear map
	
	\[
	H^{s_{1}}(\mathbb{R}^{d}) \times H^{s_{2}}(\mathbb{R}^{d}) \longrightarrow H^{s}(\mathbb{R}^{d}),
	\]
	i.e., there exists some constant $C>0$, such that for any $u \in H^{s_{1}}(\mathbb{R}^{d})$, $v \in H^{s_{2}}(\mathbb{R}^{d})$,
	
	\[
	\|uv\|_{H^{s}} \leq C \|u\|_{H^{s_{1}}} \|v\|_{H^{s_{2}}}.
	\]
	
	\end{theorem}

\subsection{Differentiation on the space of positive measures}
\begin{definition}\label{definition linear functional derivative}
	A function $F:\Mc_2(\R^d)\rightarrow\R$ is said to have a linear functional derivative if it is continuous, and there exists a function ${\delta_{\mu} F}:\Mc_2(\R^d)\x\R^d\rightarrow\R$ that is continuous for the product topology, defined as the product of the weak convergence topology on $\mathcal{M}_{2}(\mathbb{R}^{d})$ and the standard Euclidean topology on $\mathbb{R}^{d}$, such that
	\begin{itemize}
		\item for each $m \in \mathcal{M}_2(\mathbb{R}^d)$, the mapping $x \longmapsto \delta_\mu F(m, x) \in B^{d,\lambda+2}_l$,
		\item for all $m, m' \in \mathcal{M}_2(\mathbb{R}^d)$,
		\[
			F(m) - F(m')
			~=~
			\int_0^1\int_{\R^d}{\delta_{\mu} F}(\lambda m+(1-\lambda) m', x)(m - m')(dx) d\lambda.
		\]
	\end{itemize}
\end{definition}

We say that a function $u:[0,T] \times \Mc_2(\R^d) \mapsto \mathbb{R}$ is $C^{1,2}$ in $(t,m)$  if the function $u(t,m)$ is continuous in all its variables and $( \partial_t u,\partial_{x}{\delta_{\mu} u}, \partial^2_{xx}{\delta_{\mu} u}, {\delta_{\mu} u}) $ exist and are continuous in all variables.

\begin{remark}[Uniqueness of the Linear Functional Derivative]
	The uniqueness property of the linear functional derivative in Definition \ref{definition linear functional derivative} is contingent on the underlying measure space.
	\begin{itemize}
		\item On the space of probability measures, $\mathcal{P}(\mathbb{R}^d)$, the derivative $\delta_\mu F$ is only unique up to an additive constant. This is because for any constant $c \in \mathbb{R}$, the integral $\int_{\mathbb{R}^d} c  (p- p')(dx) = c(p(\mathbb{R}^d) - p'(\mathbb{R}^d)) = 0$ vanishes. Thus, $\delta_\mu F$ and $\delta_\mu F + c$ both satisfy the defining equation.
	
		\item In contrast, on the space $\mathcal{M}_2(\mathbb{R}^d)$ of finite measures (with finite second moments), the total mass $m(\mathbb{R}^d)$ is not fixed. If one adds a constant $c \in \mathbb{R}$ to a candidate derivative $\delta_\mu F$, the integral in Definition \ref{definition linear functional derivative} acquires an extra term $c (m(\mathbb{R}^d) - m'(\mathbb{R}^d))$, which is generally non-zero. 
		This ensures that the derivative $\delta_\mu F$ is uniquely determined in this setting.
	\end{itemize}
	\end{remark}

	\begin{remark}\label{rmk:regularity_assumption}
		In Definition \ref{definition linear functional derivative}, we impose that for each $m \in \mathcal{M}_2(\mathbb{R}^d)$, the linear functional derivative $\delta_\mu F(m, \cdot)$ belongs to the space $B_l^{d, \lambda+2}$, which requires certain growth and smoothness conditions (specifically, the function has at most linear growth and belongs to a Sobolev-type space with index $\lambda+2$). This is a stronger regularity assumption compared to the definition of linear functional derivatives in \cite[Definition 2.4]{CaoRenTan}, where only uniform quadratic growth is required. Our motivation for this enhancement is to ensure the continuity properties of the Hamiltonian in Chapter 5.
		
		However, this additional regularity does not affect the validity of the It\^o's formula from \cite[Theorem 2.14]{CaoRenTan}. The It\^o's formula in \cite{CaoRenTan} relies solely on the existence and continuity of the first and second-order linear and intrinsic derivatives, along with standard growth conditions (e.g., the derivatives having at most quadratic growth). Since our Definition \ref{definition linear functional derivative} implies that $\delta_\mu F(m, \cdot)$ is continuous and has at most linear growth (hence certainly quadratic growth), all conditions required by \cite[Theorem 2.14]{CaoRenTan} are satisfied. Therefore, the results from \cite{CaoRenTan} remain directly applicable in our framework.
		\end{remark}

\subsection{The non-linear PDE and its viscosity solution}

Given $G: [0, T] \times \Mc_2(\R^d) \times B^{d,\lambda+1}_l \times B^{d \times d,\lambda}_l \times B^{d,\lambda+2}_l \rightarrow\R$, our aim is to prove the existence and uniqueness for viscosity solutions of the PDE
\begin{equation}\label{eq:HJBdef}
		\begin{cases}
			-\left(\partial_t v + G\left( \cdot, \partial_{x}{\delta_{\mu} v}, \partial^2_{xx}{\delta_{\mu} v}, {\delta_{\mu} v} \right) \right)(t,m)
			~=~
			0, & \text{on } [0,T)\times\Mc_2(\R^d),\\
			v(T,m)
			~=~
			\<g(\cdot, m), m\>,  & \text{on } \Mc_2(\R^d).
		\end{cases}
	\end{equation}

\begin{definition}
For any function $u : [0,T] \times \Mc_2(\R^d) \longmapsto \mathbb{R}$, and $\theta \in [0,T) \times \mathcal{M}_2(\mathbb{R}^d)$. We define the (partial) first-order superjet $J^{1,+}u(\theta) \subset \mathbb{R} \times B^{d,\lambda+1}_l \times B^{d \times d,\lambda}_l \times B^{d,\lambda+2}_l$ of $u$ at $\theta$ by
\[
J^{1,+}u(\theta) := \left\{ \left( \partial_t\phi, \partial_{x}{\delta_{\mu} \phi}, \partial^2_{xx}{\delta_{\mu} \phi}, {\delta_{\mu} \phi} \right)(\theta) : 
\begin{array}{c} 
u - \phi \text{ has a local maximum at } \theta, \\ 
\phi: [0,T] \times \Mc_2(\R^d) \to \R \text{ is } C^{1,2}.
\end{array} 
\right\},
\]
the first-order subjet $J^{1,-}u(\theta) := -J^{1,+}(-u)(\theta)$.
\end{definition}

\begin{definition}
We say that $u$ is a \textit{viscosity subsolution} (supersolution) to \eqref{eq:HJBdef} if
\[
u : [0,T] \times \mathcal{M}_2(\mathbb{R}^d) \mapsto \R,
\]
satisfies that for $(t,m) \in [0,T) \times \mathcal{M}_2(\mathbb{R}^d)$,
\[
-b - G(t,m,p,q,r) \leq (\geq)  0, \quad \text{for any } (b,p,q,r) \in J^{1,+}{u}(t,m)  (J^{1,-}{u}(t,m)).
\]
\end{definition}

\begin{remark}
	\label{remark:lifting_approach}
	Regarding the space of probability measures, a prominent approach to analyzing HJB equations on the Wasserstein space, pioneered by Pham and Wei \cite{Pham Wei 2017}, involves "lifting" the PDE to 
	a Hilbert space.
	In this framework, a function is a viscosity solution if its lifting function is a viscosity solution to the lifted PDE under standard notion of viscosity solution in Hilbert space. 
	
	However, this powerful methodology is not directly applicable to the PDE studied in this paper. The fundamental obstacle arises from the branching mechanism: the population size evolves stochastically over time. 
	Consequently, the total mass of the branching marginal measure varies through time, and it is unclear to us how to define the lifting function.
	
	We therefore adopt an alternative methodology in \cite{Erhan He 2025} by employing a notion of viscosity solutions based on the linear functional derivatives for functionals defined  on the space of finite measures $\mathcal{M}_2(\mathbb{R}^d)$. This intrinsic approach allows us to handle the varying population size naturally.
	\end{remark}

\section{Comparison principle}\label{section3}
We equip the space \( [0,T]\times\mathcal{M}_2(\mathbb{R}^d) \) with the product topology, where \( \mathcal{M}_2(\mathbb{R}^d) \) is equipped with the bounded Lipschitz distance $\mathbf{d}_{\mathrm{BL}}$, which induces the weak convergence topology.

Let $L > 0$ be an arbitrary constant, we define the following auxiliary function:
\begin{equation}\label{aux}
\vartheta:[0,T]\times\mathcal{M}_2(\mathbb{R}^{d})\ni\theta=(t,m)\mapsto e^{-Lt} \left(\int_{\mathbb{R}^{d}}\sqrt{1+|x|^{2}}m(dx) + m^2(\R^d)\right).
\end{equation}

\begin{lemma}\label{compact}
	The auxiliary function defined in \eqref{aux} satisfies the following two properties:
\begin{enumerate}
    \item $\vartheta$ is lower semicontinuous.
    \item For each $c \in \mathbb{R}$, the set $\Sigma_c = \{\theta \in [0, T] \times \mathcal{M}_2(\mathbb{R}^d) : \vartheta(\theta) \leq c_1 + c_2m(\R^d)\}$ is compact.
\end{enumerate}
\end{lemma}
\begin{proof}
\textbf{Part 1:} 
    Let $\theta_n = (t_n, m_n)$ be a sequence in $[0, T] \times \mathcal{M}_2(\mathbb{R}^d)$ converging to $\theta = (t, m)$ in the product topology. That is, $t_n \to t$ in $[0, T]$ and $m_n \to m$ weakly in $\mathcal{M}_2(\mathbb{R}^d)$. We need to show that
\[
\vartheta(\theta) \leq \liminf_{n \to \infty} \vartheta(\theta_n).
\]
Define the function $f: \mathbb{R}^d \to \mathbb{R}$ by $f(x) = \sqrt{1 + |x|^2}$. Note that $f$ is continuous and nonnegative.

By the definition of weak convergence of measures, for any bounded continuous function $g: \mathbb{R}^d \to \mathbb{R}$, we have
\[
\int_{\mathbb{R}^d} g \, dm_n \to \int_{\mathbb{R}^d} g \, dm.
\]
Since $f$ is lower semicontinuous, there exists a sequence of bounded continuous functions $(f_k)_{k \geq 1}$ such that $f_k \uparrow f$ pointwise. For each $k$, by weak convergence,
\[
\int_{\mathbb{R}^d} f_k \, dm = \lim_{n \to \infty} \int_{\mathbb{R}^d} f_k \, dm_n \leq \liminf_{n \to \infty} \int_{\mathbb{R}^d} f \, dm_n.
\]
Then, by the monotone convergence theorem,
\[
\int_{\mathbb{R}^d} f \, dm = \lim_{k \to \infty} \int_{\mathbb{R}^d} f_k \, dm \leq \liminf_{n \to \infty} \int_{\mathbb{R}^d} f \, dm_n,
\]
and since \(t \to e^{-Lt}\) and \(m \to m^2(\R^d)\) are continuous functions, together we get
\[
\vartheta(\theta) \leq \liminf_{n \to \infty} \vartheta(\theta_n).
\]
Therefore, $\vartheta$ is lower semicontinuous.

\textbf{Part 2:} Note that if $c < 0$, then $\Sigma_c = \emptyset$ which is compact. So assume $c \geq 0$. Due to the lower-semicontinuity we proved in part 1, it can be easily checked that $\Sigma_c$ is a closed set of the set $[0, T] \times K_c$, where
\[
K_c = \left\{ m \in \mathcal{M}_2(\mathbb{R}^d) : \int_{\mathbb{R}^d} \sqrt{1 + |x|^2} \, m(dx) + m^2(\R^d) \leq e^{LT} (c_1 + c_2m(\R^d)) \right\}.
\]
Since $[0, T]$ is compact, it suffices to show that $K_c$ is compact in $\mathcal{M}_2(\mathbb{R}^d)$ with weak convergence topology. Since for any $x\in \mathbb{R}^{d}$, 
\[
\int_{\mathbb{R}^{d}}\sqrt{1+|x|^2}\,m(dx) \ge \int_{\mathbb{R}^{d}}1\,m(dx) = m(\mathbb{R}^{d}).
\] 
Therefore $\vartheta(\theta)\le  c_1 + c_2m(\R^d)$ implies $m(\mathbb{R}^{d})\le c$ for some constant c, and the space $\mathbb{R}^d$ is Polish, then we can apply Prokhorov's Theorem. Therefore, we need to show:
\begin{enumerate}
    \item $K_c$ is tight,
    \item $K_c$ is closed.
\end{enumerate}
For tightness, we need to show that for every $\eps > 0$, there exists a compact set $E_\eps \subset \mathbb{R}^d$ such that
\[
\sup_{m \in K_c} m(\mathbb{R}^d \setminus E_\eps) < \eps.
\]

Let $\eps > 0$. Recall $f(x) = \sqrt{1 + |x|^2}$. Note that for $|x| \geq 1$, we have $f(x) \geq |x|$. For any $R > 0$ and $m \in K_c$,
\[
\int_{\mathbb{R}^d} f \, dm \geq \int_{|x| > R} f(x) \, m(dx) \geq \int_{|x| > R} |x| \, m(dx) \geq R \cdot m(\{|x| > R\}).
\]
Since $m \in K_c$, we have $\int f \, dm \leq e^{LT}(c_1 + c_2m(\R^d)) - m^2(\R^d)$, so
\[
R \cdot m(\{|x| > R\}) \leq \int_{\mathbb{R}^d} f \, dm \leq e^{LT}(c_1 + c_2m(\R^d)) - m^2(\R^d).
\]
Thus,
\[
m(\{|x| > R\}) \leq \frac{(c_1 + c_2c) e^{LT}}{R}.
\]
Choose $R > \max\left\{1, \frac{(c_1 + c_2c) e^{LT}}{\eps}\right\}$ so that $\frac{(c_1 + c_2c) e^{LT}}{R} < \eps$. Then, for $E_\eps = \{ x \in \mathbb{R}^d : |x| \leq R \}$,
\[
m(\mathbb{R}^d \setminus E_\eps) = m(\{|x| > R\}) < \eps, \quad \forall m \in K_c.
\]
Therefore, $K_c$ is tight.

For closeness, Let $\{m_n\}_{n \geq 1}$ be a sequence in $K_c$ converging weakly to some $m \in \mathcal{M}_2(\mathbb{R}^d)$. We need to show that $m \in K_c$. Since $m_n \to m$ weakly and $f$ is lower semicontinuous and nonnegative, by the same argument as in the lower semicontinuity proof, we have
\[
\int_{\mathbb{R}^d} f \, dm \leq \liminf_{n \to \infty} \int_{\mathbb{R}^d} f \, dm_n \leq e^{LT} (c_1 + c_2m(\R^d)) - m^2(\R^d).
\]
Thus, $m \in K_c$, so $K_c$ is closed.

In conclusion, by Prokhorov's theorem, $\Sigma_c$ is a closed subset of $[0, T] \times K_c$ is compact in the product topology.
\end{proof}

\begin{lemma}\label{ishii's lemma}
Suppose ${u}, -{v}: [0,T] \times \mathcal{M}_{2}(\mathbb{R}^{d}) \to \mathbb{R}$ are upper semi-continuous functions satisfying \(|u(t,m)| + |v(t,m)| \leq  C (1+m(\R^d))\) for some constant $C$. For any $\delta > 0$, introduce
\begin{align*}
\tilde{u}: [0,T] \times \mathcal{M}_{2}(\mathbb{R}^{d}) &\ni (t, m) \mapsto u(t, m) - \delta \vartheta(t, m), \\
\tilde{v}: [0,T] \times \mathcal{M}_{2}(\mathbb{R}^{d}) &\ni (s, n) \mapsto v(s, n) + \delta \vartheta(s, n).
\end{align*}
Then there exists a local maximum $(\theta^{*}, \iota^{*}) = ((t^{*}, m^{*}), (s^{*}, n^{*}))$ of
\begin{align}\label{max_tilde}
 (\theta, \iota) \in \left([0,T] \times \mathcal{M}_2 (\mathbb{R}^d)\right)^2 \longmapsto \tilde{u}(\theta) - \tilde{v}(\iota) - \frac{1}{2\eps} \left( |t - s|^{2}+ \rho_{F}^{2}(m, n) \right).    
\end{align}
Assume $\theta^{*}, \iota^{*}$ are in the interior $[0, T) \times \mathcal{M}_{2}(\mathbb{R}^{d}) \times \mathbb{R}^{d}$, and for $\eta\in \mathcal{M}_{2}(\mathbb{R}^{d})$ denote
\begin{align*}
\mathcal{L}(\eta, m, n) &:= 2 \int \frac{\operatorname{Re}(F_{k}(\eta) (F_{k}(m) - F_{k}(n))^{*})}{(1 + |k|^{2})^{\lambda}} \, dk, \\
\Phi(m) &:= 2 \rho_{F}^{2}(m, m^{*}) + \mathcal{L}(m, m^{*}, n^{*}), \\
\Psi(n) &:= 2 \rho_{F}^{2}(n, n^{*}) - \mathcal{L}(n, m^{*}, n^{*}). 
\end{align*}
Then 
\begin{align*}
&\left( \frac{1}{\eps}(t^{*} - s^{*}), \frac{\partial_{x}{\delta_{\mu}} \Phi(m^{*})(\cdot)}{2\eps}, \frac{\partial^2_{xx}{\delta_{\mu}} \Phi(m^{*})(\cdot)}{2\eps}, \frac{{\delta_{\mu}} \Phi(m^{*})(\cdot)}{2\eps} \right) \in {J}^{1,+} \tilde{u}(\theta^{*}), \\
&\left( \frac{1}{\eps}(t^{*} - s^{*}), -\frac{\partial_{x}{\delta_{\mu}} \Psi(n^{*})(\cdot)}{2\eps}, -\frac{\partial^2_{xx}{\delta_{\mu}} \Psi(n^{*})(\cdot)}{2\eps}, -\frac{{\delta_{\mu}} \Psi(n^{*})(\cdot)}{2\eps} \right) \in {J}^{1,-} \tilde{v}(\iota^{*}).
\end{align*}
\end{lemma}

\begin{proof}
Denote the maximum value of \eqref{max_tilde} obtained by $a \in \mathbb{R}$. We consider the set $A$ of $\theta$, $\iota$ such that
\[
\tilde{u}(\theta) - \tilde{v}(\iota) - \frac{1}{2\eps}\left(|t-s|^{2} + \rho_{F}^{2}(m,n)\right) \geq a - 1,
\]
and hence
\[
\delta\left(\vartheta(\theta) + \vartheta(\iota)\right) \leq 1 - a + C m(\R^d).
\]
Therefore by Lemma \ref{compact}, the closure of $A$ is compact, and we denote by $(\theta^{*},\iota^{*}) = ((t^{*},m^{*}),(s^{*},n^{*}))$ a global maximum of \eqref{max_tilde}. We assume that $t^{*}, s^{*} < T$.

Define the functions
\[
u_{1}(t, m):=\tilde{u}(t, m)-\frac{1}{2\eps}\Phi(m),
\]
\[
v_{1}(s, n):=\tilde{v}(s, n)+\frac{1}{2\eps}\Psi(n).
\]
Given any Hilbert space $(L,\langle\cdot\rangle)$ and $a,b,c,d\in L$, it is straightforward that
\[
|a-b|^{2}+|c-d|^{2}-2\langle a-b,c-d\rangle=|a-b-(c-d)|^{2}\leq 2(|a-c|^{2}+|b-d|^{2}).
\]
Therefore we have
\begin{equation}\label{inner-ineq}
\begin{aligned}
|F_{k}(m)-F_{k}(n)|^{2}+|F_{k}(m^{*})-F_{k}(n^{*})|^{2} &\leq 2(|F_{k}(m)-F_{k}(m^{*})|^{2}+|F_{k}(n)-F_{k}(n^{*})|^{2}) \\
&\quad + 2 \operatorname{Re}((F_{k}(m)-F_{k}(n))(F_{k}(m^{*})-F_{k}(n^{*}))).
\end{aligned}
\end{equation}
Hence it holds that
\[
2\rho_{F}^{2}(m,m^{*})+2\rho_{F}^{2}(n,n^{*})+\mathcal{L}(m,m^{*},n^{*})-\mathcal{L}(n,m^{*},n^{*})\geq\rho_{F}^{2}(m,n)+\rho_{F}^{2}(m^{*},n^{*}).
\]
Thus, we have the inequality,
\begin{align*}
&u_{1}(t, m)-v_{1}(s, n)-\frac{1}{2\eps}|t-s|^{2} \\
&=\tilde{u}(t, m)-\tilde{v}(s, n)-\frac{1}{2\eps}|t-s|^{2} \\
&\quad -\frac{1}{2\eps}\left(2\rho_{F}^{2}(m,m^{*})+2\rho_{F}^{2}(n,n^{*})\right)-\frac{1}{2\eps}\left(\mathcal{L}(m,m^{*},n^{*})-\mathcal{L}(n,m^{*},n^{*})\right) \\
&\leq\tilde{u}(t, m)-\tilde{v}(s, n)-\frac{1}{2\epsilon}|t-s|^{2}-\frac{1}{2\eps}(\rho_{F}^{2}(m,n)+\rho_{F}^{2}(m^{*},n^{*})) \\
&=\tilde{u}(\theta)-\tilde{v}(\iota)-\frac{1}{2\epsilon} d_{F}^{2}(\theta,\iota)-\frac{1}{2\epsilon}\rho_{F}^{2}(m^{*},n^{*}).
\end{align*}
Furthermore, for $(m,n)=(m^{*},n^{*})$, the inequality \eqref{inner-ineq} is an equality. Thus, the function
\begin{equation*}
    u_{1}(t, m)-v_{1}(s,n)-\frac{1}{2\epsilon}|t-s|^{2} 
\end{equation*}
admits the strict global maximum at $((t^{*},m^{*}),(s^{*},n^{*}))$, which means that
\[
u_1(t, m) - v_1(s, n) - \frac{1}{2\eps} |t - s|^2 \leq u_1(t^*, m^*) - v_1(s^*, n^*) - \frac{1}{2\eps} |t^* - s^*|^2,
\]
which is equivalent to
\[
u_1(t, m) - v_1(s, n) \leq u_1(t^*, m^*) - v_1(s^*, n^*) + \frac{1}{2\eps} |t - s|^2 - \frac{1}{2\eps} |t^* - s^*|^2.
\]
Consider another function
\begin{equation*}
    F(t, s, m, n) ~=~ u_{1}(t, m)-v_{1}(s,n)-\frac{1}{\epsilon}(t-s^*)(t-t^*) - \frac{1}{\epsilon}(s-t^*)(s-s^*).
\end{equation*}
Then we have 
\begin{align*}
    F(t,s,m,n) \leq u_1(t^*, m^*) - v_1(s^*, n^*) + \frac{1}{2\eps} |t - s|^2 - \frac{1}{2\eps} |t^* - s^*|^2 -\frac{1}{\eps}(t-s^*)(t-t^*) - \frac{1}{\eps}(s-t^*)(s-s^*).
\end{align*}
At the point \((t^*,m^*,s^*,n^*)\), we have
\[
F(t^*, s^*, m^*, n^*) = u(t^*, m^*) - v(s^*, n^*).
\]
Define
\[
E(t, s) = \frac{1}{2} |t - s|^2 - \frac{1}{2} |t^* - s^*|^2 - (t - s^*)(t - t^*) - (s - t^*)(s - s^*).
\]
Upon some direct algebraic computation we shall get
\[
E(t, s) = -\frac{1}{2} (t + s)^2 + (t + s)(t^* + s^*) - \frac{1}{2} (t^* + s^*)^2 = -\frac{1}{2} \left[ (t + s) - (t^* + s^*) \right]^2 \leq 0.
\]
Then we have
\[
F(t, s, m, n) \leq u(t^*, m^*) - v(s^*, n^*) + E(t, s) \leq u(t^*, m^*) - v(s^*, n^*) = F(t^*, s^*, m^*, n^*).
\]
Therefore, $F(t, s, m, n)$ achieves a maximum at $(t^*, m^*)$ and $(s^*, n^*)$.\\
Denoting

\begin{align*}
\Phi_1(\theta) &= \frac{1}{2\eps}\Phi(m)+\frac{1}{\eps}(t-s^*)(t-t^*), \\
\Psi_1(\iota) &= \frac{1}{2\eps}\Psi(n)+\frac{1}{\eps}(s-t^*)(s-s^*).
\end{align*}
By the definition of $u_{1},v_{1}$, the function $\tilde{u}(\theta)-\tilde{v}(\iota)-\Phi_1(\theta)-\Psi_1(\iota)$ has a local maximum at $(\theta^{*},\iota^{*})$. Therefore, according to the definition of jets of $\tilde{u},\tilde{v}$, we have

\begin{align*}
&(\partial_{t},\partial_{x}{\delta_{\mu}},\partial^2_{xx}{\delta_{\mu}},\delta_{\mu})\Phi_1(\theta^{*}) \\
&\quad = \left(\frac{1}{\eps}(t^{*}-s^{*}),\frac{\partial_{x}{\delta_{\mu}}\Phi(m^{*})(\cdot)}{2\eps},\frac{\partial^2_{xx}{\delta_{\mu}}\Phi(m^{*})(\cdot)}{2\eps},\frac{{\delta_{\mu}}\Phi(m^{*})(\cdot)}{2\eps}\right) \in {J}^{1,+}\tilde{u}(\theta^{*}), \\
&(\partial_{t},\partial_{x}{\delta_{\mu}},\partial^2_{xx}{\delta_{\mu}},\delta_{\mu})\Psi_1(\iota^{*}) \\
&\quad = \left(\frac{1}{\eps}(t^{*}-s^{*}),-\frac{\partial_{x}{\delta_{\mu}}\Psi(n^{*})(\cdot)}{2\eps},-\frac{\partial^2_{xx}{\delta_{\mu}}\Psi(n^{*})(\cdot)}{2\eps},-\frac{{\delta_{\mu}}\Psi(n^{*})(\cdot)}{2\eps}\right) \in {J}^{1,-}\tilde{v}(\iota^{*}).
\end{align*}

\end{proof}

\begin{assumption}\label{A.ham}
\begin{enumerate}
\item[(i)] The Hamiltonian $G$ satisfies, with a constant $L_{G} > 0$,
\[
	\big|G(\theta,p_{1},q_{1},r_{1})-G(\theta,p_{2},q_{2},r_{2}) \big|
	\leq 
	L_{G}\left(\int_{\mathbb{R}^{d}}(1+|x|)m(dx)\right)
\left(|p_{1}-p_{2}|_{l}+|q_{1}-q_{2}|_{l}+|r_{1}-r_{2}|_{l}\right),
\]
for all \(\theta = (t,m) \in [0,T) \times \mathcal{M}_2(\R^d)\), $p_{1},p_{2},r_1,r_2\in B^{d}_l$ and $q_{1},q_{2}\in B^{d\times d}_l$.

\item[(ii)] There exists a modulus of continuity $\omega_{G}$ such that for all $\eps>0$, $(\theta=(t,m),\iota=(s,n))\in[0,T)\times\mathcal{M}_{2}(\mathbb{R}^{d})$, we have the inequality
\[
G(\theta,\nabla\kappa,\nabla^{2}\kappa,\kappa)-G(\iota,\nabla\kappa,\nabla^{2}\kappa,\kappa)
\]
\[
\leq\omega_{G}\left((1 + m(\R^d))(\frac{1}{\eps}d_{F}^{2}(\theta,\iota)+d_{F}(\theta,\iota))\right)\left(\int_{\mathbb{R}^{d}}1+|x|(m+n)(dx)\right),
\]
where the function $\kappa$ on $\mathbb{R}$ is defined as
\[
\kappa(x):=\frac{1}{\eps}\int_{\mathbb{R}^{d}}\frac{\operatorname{Re}(F_{k}(m-n)f_{k}^{*}(x))}{(1+|k|^{2})^{\lambda}}dk,\quad x\in\mathbb{R}^{d}.
\]
\end{enumerate}
\end{assumption}
\begin{remark}
	By a direct computaion, it can be proved that \(\kappa(x) \in B^{d,s}_l\) for any \(s \leq 2\lambda - d/2\).
\end{remark}

\begin{theorem}\label{general comparison}
	Let Assumption \ref{A.ham} hold true,
	and \( u, -v : [0,T] \times \mathcal{M}_2(\mathbb{R}^d) \to \mathbb{R} \) be upper semi-continuous functions such that, with some constant $C > 0$,  \(|u(t,m)| + |v(t,m)| <  C (1 + m(\R^d))\) for all $m \in \mathcal{M}_2(\mathbb{R}^d)$.
	Assume that \( u \) (resp. \( v \)) is a viscosity subsolution (resp. supersolution) of the equation
    \begin{align}\label{eq:bellman}
    -{\partial_t v}(t, m) 
	- G(t, m, \partial_{x}{\delta_{\mu} v}, \partial^2_{xx}{\delta_{\mu} v}, {\delta_{\mu} v}) (t,m)
	&= 0 \quad \text{on } [0, T) \times \Mc_2(\R^d),
\end{align}
	such that \( u(T,\cdot) \leq v(T,\cdot) \).
    Then \( u \leq v \) for all \( (t,m) \in [0,T] \times \mathcal{M}_2(\mathbb{R}^d) \).
\end{theorem}

\begin{proof}
\textbf{Step 1: Preliminary.} 
We prove the result by contradiction. Assume that
\[
\sup_{(t,m) \in [0,T] \times \mathcal{M}_{2}(\mathbb{R}^{d})} ({u} - {v})(t,m) > 0.
\]
Then we choose a \( h,r> 0 \) small enough such that \( {u}_{h} := {u} - h(T - t + 1) \) satisfies
\[
\sup_{(t,m) \in [0,T] \times \mathcal{M}_{2}(\mathbb{R}^{d})} ({u}_{h} - {v})(t,m) \geq 3r > 0.
\]
We then verify that \( {u}_{h} \) is still a viscosity subsolution of the equation \eqref{eq:bellman}. Suppose \( \phi \) is a \( C^{1,2} \) test function such that \({u}_{h} - \phi \) has a local maximum at \( (t^{*},m^{*}) \). Equivalently,
\(
{u} - (\phi + h(T - t + 1)) \)  has a local maximum at \( (t^{*}, \mu^{*}, m^{*}). \)
We write $\phi + h(T - t + 1)$  as  $\phi_{h}$ for short.

Then, the viscosity property of \( u \) gives that, for any \( (t^{*}, m^{*}) \),
\begin{align*}
0 > &-\left( \partial_{t}\phi_{h} +  G(\cdot, \partial_{x}{\delta \phi_h}, \partial^2_{xx}{\delta \phi_h}, {\delta \phi_h} \right)(t^{*}, m^{*}) \\
&= -\left( \partial_{t}\phi_{h} + G(\cdot, \partial_{x}{\delta \phi}, \partial^2_{xx}{\delta \phi}, {\delta \phi} \right)(t^{*}, m^{*}) \\
&= -\left( \partial_{t}\phi + G(\cdot, \partial_{x}{\delta \phi}, \partial^2_{xx}{\delta \phi}, {\delta \phi} \right)(t^{*}, m^{*}) + h.
\end{align*}
and thus
\[
	-\left( \partial_{t}\phi + G(\cdot, \partial_{x}{\delta \phi}, \partial^2_{xx}{\delta \phi}, {\delta \phi} \right)(t^{*}, m^{*}) \leq -h.
\]
\textbf{Step 2: Doubling the variables.} Recall that for each $\theta = (t, m) \in [0, T] \times \mathcal{M}_{2}(\mathbb{R}^{d})$, the auxiliary function defined in \eqref{aux},
\[
\vartheta(\theta) = e^{-Lt} \left(\int_{\mathbb{R}^{d}}\sqrt{1+|x|^{2}}m(dx) + m^2(\R^d)\right),
\] 
is lower-semicontinuous by Lemma \ref{compact}.
Now for any $\eps, \delta > 0$, and $(\theta, \iota) \in \left( [0, T] \times \mathcal{M}_{2}(\mathbb{R}^{d}) \right)^{2}$, we define
\[
H_{\eps}^{\delta}(\theta, \iota) := {u}_{h}(\theta) - {v}(\iota) - \frac{1}{2\eps} d_{F}^{2}(\theta, \iota) - \delta (\vartheta(\theta) + \vartheta(\iota)).
\]
Since $u,-v$ and $-\vartheta$ are upper-semicontinuous functions, while the metric $d_F$ is continuous, $H_{\eps}^{\delta}$ is also an upper-semicontinuous function. Then \textbf{we claim that the supremum of $H_{\eps}^{\delta}$ over all admissible $(\theta, \iota)$ can be obtained.}
In fact, there exists an element $\theta_{0} \in [0, T] \times \mathcal{M}_{2}(\mathbb{R}^{d})$, such that
\[
H_{\eps}^{\delta}(\theta_{0}, \theta_{0}) + 2\delta \vartheta(\theta_{0}) = {u}_{h}(\theta_{0}) - {v}(\theta_{0}) \geq 2r > 0,
\]
and for $\delta < \frac{r}{2\vartheta(\theta_{0}) + 1}$, we have
\[
\sup_{(\theta, \iota) \in \left( [0, T] \times \mathcal{M}_{2}(\mathbb{R}^{d})  \right)^{2}} H_{\eps}^{\delta}(\theta, \iota) \geq H_{\eps}^{\delta}(\theta_{0}, \theta_{0}) \geq r > 0.
\]
Take a sequence $\{(\theta_{k}, \iota_{k})\}_{k \in \mathbb{N}_{+}}=\{((t_k,m_k),(s_k,n_k))\}_{k \in \mathbb{N}_{+}}$ with $H_{\eps}^{\delta}(\theta_{k}, \iota_{k}) > 0$ such that
\[
\lim_{k \to \infty} H_{\eps}^{\delta}(\theta_{k}, \iota_{k}) = \sup_{(\theta, \iota) \in \left( [0, T] \times \mathcal{M}_{2}(\mathbb{R}^{d}) \right)^{2}} H_{\eps}^{\delta}(\theta, \iota).
\]
We have the estimation that for each $k \in \mathbb{N}_{+}$
\[
\begin{array}{rl}
\delta (\vartheta(\theta_{k}) + \vartheta(\iota_{k})) & \leq {u}_{h}(\theta_{k}) - {v}(\iota_{k}) - \frac{1}{2\eps} d_{F}^{2}(\theta_{k}, \iota_{k}) \\
& \leq {u}_{h}(\theta_{k}) - {v}(\iota_{k}) \leq C (1+m_k(\R^d)+n_k(\R^d)).
\end{array}
\]
 Since, for any $c>0$, by Lemma \ref{compact}, the sublevel set $\{\theta \in [0,T] \times \mathcal{M}_2(\mathbb{R}^d) : \vartheta(\theta) \leq c m(\R^d)\}$ is compact. Without loss of generality we can assume that $\{(\theta_k, \iota_k)\}_{k \in \mathbb{N}_+}$ is a convergent sequence in $[0,T] \times \mathcal{M}_2(\mathbb{R}^d)$, and

\[
\lim_{k \to \infty} (\theta_k, \iota_k) = (\theta_\eps^{\delta}, \iota_\eps^{\delta})
\]
for some $(\theta_\eps^{\delta}, \iota_\eps^{\delta}) \in ([0,T] \times \mathcal{M}_2(\mathbb{R}^d))^{2}$. As a consequence, the upper-semicontinuity of $H_\eps^{\delta}$ yields that

\[
H_\eps^{\delta}(\theta_\eps^{\delta}, \iota_\eps^{\delta}) = \sup_{(\theta,\iota) \in ([0,T] \times \mathcal{M}_2(\mathbb{R}^d))^2} H_\eps^{\delta}(\theta,\iota) > 0,
\]

\[
\delta \big( \vartheta(\theta_\eps^{\delta}) + \vartheta(\iota_\eps^{\delta}) \big) \leq \varliminf_{k \to \infty} \delta \big( \vartheta(\theta_k) + \vartheta(\iota_k) \big) \leq C (1+m_\eps^{\delta}(\R^d)+n_\eps^{\delta}(\R^d)).
\]
Again, WLOG we can assume that $\{(\theta_\eps^{\delta}, \iota_\eps^{\delta})\}_{\eps \in \mathbb{N}_+}$ is a convergent sequence in $[0,T] \times \mathcal{M}_2(\mathbb{R}^d)$,

\[
\lim_{\eps \to 0+} (\theta_\eps^{\delta}, \iota_\eps^{\delta}) = (\theta^{\delta}, \iota^{\delta})
\]
for some $(\theta^{\delta},\iota^{\delta})\in([0,T]\times\mathcal{M}_{2}(\mathbb{R}^{d}))^{2}$.

\textbf{Then, we claim that}
\begin{equation}
\lim_{\eps \rightarrow 0}\frac{1}{2\eps} d_{F}^{2}\left(\theta_{\eps}^{\delta},\iota_{\eps}^{\delta}\right)=0.
\end{equation}
In fact, one may observe that
\[
\begin{array}{l}
\varlimsup_{\eps\rightarrow 0+}\dfrac{1}{2\eps} d_{F}^{2}\left(\theta_{\eps}^{\delta},\iota_{\eps}^{\delta}\right)=\varlimsup_{\eps\rightarrow 0+}\left(-H_{\eps}^{\delta}\left(\theta_{\eps}^{\delta},\iota_{\eps}^{\delta}\right)-\delta\left(\vartheta\left(\theta_{\eps}^{\delta}\right)+\vartheta\left(\iota_{\eps}^{\delta}\right)\right)+{u}_{h}\left(\theta_{\eps}^{\delta}\right)-{v}\left(\iota_{\eps}^{\delta}\right)\right)\\
\leq\varlimsup_{\eps\rightarrow 0+}\left({u}_{h}\left(\theta_{\eps}^{\delta}\right)-{v}\left(\iota_{\eps}^{\delta}\right)\right)\leq 2M,\\
\end{array}
\]
and obtain that
\[
\varlimsup_{\eps\rightarrow 0+}d_{F}^{2}\left(\theta_{\eps}^{\delta},\iota_{\eps}^{\delta}\right)=0,\text{ i.e. }d_{F}^{2}\left(\theta^{\delta},\iota^{\delta}\right)=0,\text{ or }\theta^{\delta}=\iota^{\delta}.
\]
Then it follows that
\[
\begin{array}{l}
\varlimsup_{\eps\rightarrow 0+}\dfrac{1}{2\eps} d_{F}^{2}\left(\theta_{\eps}^{\delta},\iota_{\eps}^{\delta}\right)=\varlimsup_{\eps\rightarrow 0+}\left(-H_{\eps}^{\delta}\left(\theta_{\eps}^{\delta},\iota_{\eps}^{\delta}\right)-\delta\left(\vartheta\left(\theta_{\eps}^{\delta}\right)+\vartheta\left(\iota_{\eps}^{\delta}\right)\right)+{u}_{h}\left(\theta_{\eps}^{\delta}\right)-{v}\left(\iota_{\eps}^{\delta}\right)\right)\\
\leq\varlimsup_{\eps\rightarrow 0+}\left(-\sup _{\theta\in[0, T]\times\mathcal{M}_{2}\left(\mathbb{R}^{d}\right)} H_{\eps}^{\delta}(\theta,\theta)-\delta\left(\vartheta\left(\theta_{\eps}^{\delta}\right)+\vartheta\left(\iota_{\eps}^{\delta}\right)\right)+{u}_{h}\left(\theta_{\eps}^{\delta}\right)-{v}\left(\iota_{\eps}^{\delta}\right)\right)\\
\leq-\sup _{\theta\in[0, T]\times\mathcal{M}_{2}\left(\mathbb{R}^{d}\right)} H_{\eps}^{\delta}(\theta,\theta)-\delta\left(\vartheta\left(\theta^{\delta}\right)+\vartheta\left(\iota^{\delta}\right)\right)+{u}_{h}\left(\theta^{\delta}\right)-{v}\left(\iota^{\delta}\right)\\
=-\sup _{\theta\in[0, T]\times\mathcal{M}_{2}\left(\mathbb{R}^{d}\right)} H_{\eps}^{\delta}(\theta,\theta)+H_{\eps}^{\delta}\left(\theta^{\delta},\theta^{\delta}\right)\leq 0,\\
\end{array}
\]
where the second inequality holds by that $\sup_{\theta\in[0, T]\times\mathcal{M}_{2}(\mathbb{R}^{d})}H_{\eps}^{\delta}(\theta,\theta)$ is independent of $\eps$, and that ${u}_{h},{v}$, and $-\vartheta$ are upper semi-continuous.
Moreover, it is clear that $t^{\delta}\neq T$, where $\theta^{\delta}=(t^{\delta},m^{\delta})$. Otherwise, it follows that
\[
0<r\leq\varlimsup_{\eps\rightarrow 0+}H_{\eps}^{\delta}\left(\theta_{\eps}^{\delta},\iota_{\eps}^{\delta}\right)\leq\varlimsup_{\eps\rightarrow 0+}{u}_{h}\left(\theta_{\eps}^{\delta}\right)-{v}\left(\iota_{\eps}^{\delta}\right)\leq{u}_{h}\left(\theta^{\delta}\right)-{v}\left(\iota^{\delta}\right)\leq 0.
\]
\textbf{Step 3: Contradiction by viscosity property and estimation} Define $\tilde{u}:={u}_{h}-\delta\vartheta$, $\tilde{v}:={v}+\delta\vartheta$. Without loss of generality, we assume that $(\theta_{\eps}^{\delta},\iota_{\eps}^{\delta})$ is a strict global maximum of
\[
(\theta,\iota) \longmapsto H_{\eps}^{\delta}(\theta,\iota) = \tilde{u}(\theta) - \tilde{v}(\iota) - \frac{1}{2\eps}d_{F}^{2}(\theta,\iota).
\]
Then by Ishii's Lemma \ref{ishii's lemma}, it follows that for any $\eps > 0$,
\begin{align*}
\left( \frac{1}{\eps}(t^{*} - s^{*}), \frac{\partial_{x}{\delta_{\mu}} \Phi(m^{*})(\cdot)}{2\eps}, \frac{\partial^2_{xx}{\delta_{\mu}} \Phi(m^{*})(\cdot)}{2\eps}, \frac{{\delta_{\mu}} \Phi(m^{*})(\cdot)}{2\eps} \right) & \in {J}^{1,+}\tilde{u}(\theta^{*}), \\
\left( \frac{1}{\eps}(t^{*} - s^{*}), -\frac{\partial_{x}{\delta_{\mu}} \Psi(n^{*})(\cdot)}{2\eps}, -\frac{\partial^2_{xx}{\delta_{\mu}} \Psi(n^{*})(\cdot)}{2\eps}, -\frac{{\delta_{\mu}} \Psi(n^{*})(\cdot)}{2\eps} \right) & \in {J}^{1,-}\tilde{v}(\iota^{*}).
\end{align*}
For a better statement, let us recall and define some notation:
\begin{align*}
\theta_{\eps}^{\delta} &= (t_{\eps}^{\delta},m_{\eps}^{\delta}), \\
\iota_{\eps}^{\delta} &= (s_{\eps}^{\delta},n_{\eps}^{\delta}), \\
q(x) &:= \sqrt{1 + |x|^{2}}, \quad x \in \mathbb{R}^{d}, \\
\kappa_{\eps}^{\delta}(x) &:= \frac{1}{\eps} \int_{\mathbb{R}^{d}} \frac{\operatorname{Re}(F_{k}(m_{\eps}^{\delta} - n_{\eps}^{\delta}) f_{k}^{*}(x))}{(1 + |k|^{2})^{\lambda}} \, dk, \quad x \in \mathbb{R}^{d}.
\end{align*}
The direct calculation yields the result that for $\theta = (t, m) \in [0, T] \times \mathcal{M}_{2}(\mathbb{R}^{d})$,

\[
\partial_{x}{\delta_{\mu}}\vartheta(\theta, x) = e^{-Lt}\nabla q(x), \quad \partial^2_{xx}{\delta_{\mu}}\vartheta(\theta, x) = e^{-Lt}\nabla^{2} q(x), \quad \delta_{\mu}\vartheta(\theta) = e^{-Lt}(q(x) + 2m(\R^d)),
\]
\[
|\nabla q|_l = \sup_{x \in \mathbb{R}^d} \frac{\frac{|x|}{\sqrt{1 + |x|^2}}}{1 + |x|} = \frac{\sqrt{2}}{4}, \quad 
|\nabla^2 q|_l = \sup_{x \in \mathbb{R}^d} \frac{ |\nabla^2 q(x)| }{ 1 + |x| } = \sqrt{d}, \quad |q|_l = \sup_{x \in \mathbb{R}^d} \frac{\sqrt{1 + |x|^2}}{1 + |x|} = 1,
\]
\[
\frac{\partial_{x}{\delta_{\mu}}\Phi(m^{\delta}_\eps)(\cdot)}{2\eps} = -\frac{\partial_{x}{\delta_{\mu}}\Psi(n^{\delta}_\eps)(\cdot)}{2\eps} = \nabla \kappa_{\eps}^{\delta}(\cdot), \quad \frac{\partial^2_{xx}{\delta_{\mu}}\Phi(m^{\delta}_\eps))(\cdot)}{2\eps} = -\frac{\partial^2_{xx}{\delta_{\mu}}\Psi(\nu^{*})(\cdot)}{2\eps} = \nabla^{2} \kappa_{\eps}^{\delta}(\cdot),
\]
\[
\frac{{\delta_{\mu}}\Phi(m^{\delta}_\eps)(\cdot)}{2\eps} = -\frac{{\delta_{\mu}}\Psi(n^{\delta}_\eps)(\cdot)}{2\eps} = \kappa_{\eps}^{\delta}(\cdot).
\]
For convenience, we denote by
\[
\begin{array}{ll}
\alpha_{\eps}^{\delta} &:=\bigg(\nabla\kappa_{\eps}^{\delta}+\delta e^{-Lt_{\eps}^{\delta}}\nabla q,\nabla^{2}\kappa_{\eps}^{\delta}+\delta e^{-Lt_{\eps}^{\delta}}\nabla^{2}q,\kappa_{\eps}^{\delta}+ \delta e^{-Lt_{\eps}^{\delta}}q + 2m_{\eps}^{\delta}(\R^d)\bigg), \\
\beta_{\eps}^{\delta} &:=\bigg(\nabla\kappa_{\eps}^{\delta}-\delta e^{-Ls_{\eps}^{\delta}}\nabla q,\nabla^{2}\kappa_{\eps}^{\delta}-\delta e^{-Ls_{\eps}^{\delta}}\nabla^{2}q,\kappa_{\eps}^{\delta}-\delta e^{-Ls_{\eps}^{\delta}}q - 2n_{\eps}^{\delta}(\R^d)\bigg).
\end{array}
\]
Then, the mapping $\theta\longmapsto(\partial_{x}{\delta_{\mu}}\vartheta,\partial^2_{xx}{\delta_{\mu}}\vartheta,\delta_{\mu}\vartheta)(\theta)$ is continuous, and it follows from the linearity and the definition of the closure of the jets that
\[
\bigg(\frac{1}{\eps}(t_{\eps}^{\delta}-s_{\eps}^{\delta}) - \delta L\vartheta(\theta_{\eps}^{\delta}),\alpha_{\eps}^{\delta}\bigg)\in{J}^{1,+}{u}_{h}(\theta_{\eps}^{\delta}),\quad\bigg(\frac{1}{\eps}(t_{\eps}^{\delta}-s_{\eps}^{\delta}) + \delta L\vartheta(\iota_{\eps}^{\delta}),\beta_{\eps}^{\delta}\bigg)\in{J}^{1,-}{v}(\iota_{\eps}^{\delta}).
\]
By the viscosity property of ${u}_{h}$, and ${v}$, one has that
\[
-\frac{1}{\eps}(t_{\eps}^{\delta}-s_{\eps}^{\delta})+\delta L\vartheta(\theta_{\eps}^{\delta})-G\bigg(\theta_{\eps}^{\delta},\alpha_{\eps}^{\delta}\bigg)\leq-h,\quad-\frac{1}{\eps}(t_{\eps}^{\delta}-s_{\eps}^{\delta})-\delta L\vartheta(\iota_{\eps}^{\delta})-G\bigg(\iota_{\eps}^{\delta},\beta_{\eps}^{\delta}\bigg)\geq 0,
\]
implying,
\[
h\leq G\bigg(\theta_{\eps}^{\delta},\alpha_{\eps}^{\delta}\bigg)-G\bigg(\iota_{\eps}^{\delta},\beta_{\eps}^{\delta}\bigg) - \delta L\vartheta(\theta_{\eps}^{\delta}) - \delta L\vartheta(\iota_{\eps}^{\delta}),
\]
and on the other hand, one can obtain the following estimation
\begin{align*}
&G\left(\theta_{\eps}^{\delta},\alpha_{\eps}^{\delta}\right)-G\left(\iota_{\eps}^{\delta},\beta_{\eps}^{\delta}\right)  - \delta L\vartheta(\theta_{\eps}^{\delta}) - \delta L\vartheta(\iota_{\eps}^{\delta})\\
\leq &G\left(\theta_{\eps}^{\delta},\nabla\kappa_{\eps}^{\delta},\nabla^{2}\kappa_{\eps}^{\delta},\kappa_{\eps}^{\delta}\right)-G\left(\iota_{\eps}^{\delta},\nabla\kappa_{\eps}^{\delta},\nabla^{2}\kappa_{\eps}^{\delta},\kappa_{\eps}^{\delta}\right)\\
+ &L_{G}\!\left(\int_{\mathbb{R}^{d}}1+|x|m_{\eps}^{\delta}(dx)\right)(C + 2m_{\eps}^{\delta}(\R^d))e^{-Lt_{\eps}^{\delta}}\delta
+L_{G}\!\left(\int_{\mathbb{R}^{d}}1+|x| n_{\eps}^{\delta}(dx)\right)(C + 2n_{\eps}^{\delta}(\R^d))e^{-Ls_{\eps}^{\delta}}\delta\\
&\quad - \delta Le^{-Lt_{\eps}^{\delta}} (\int_{\mathbb{R}^{d}}\sqrt{1+|x|^{2}}m_{\eps}^{\delta}(dx) + (m_{\eps}^{\delta})^2(\R^d))
- \delta Le^{-Ls_{\eps}^{\delta}} (\int_{\mathbb{R}^{d}}\sqrt{1+|x|^{2}}n_{\eps}^{\delta}(dx) + (n_{\eps}^{\delta})^2(\R^d))\\
\leq & \omega_{G}\!\left((1 + m_{\eps}^{\delta}(\R^d))(\frac{1}{\eps}d_{F}^{2}(\theta_{\eps}^{\delta},\iota_{\eps}^{\delta})+d_{F}(\theta_{\eps}^{\delta},\iota_{\eps}^{\delta}))\right)\left(\int_{\mathbb{R}^{d}}1+|x|(m_{\eps}^{\delta}+n_{\eps}^{\delta})(dx)\right) \\
&\quad + 2\left(L_{G}(C + 2m_{\eps}^{\delta}(\R^d))e^{-Lt_{\eps}^{\delta}}\delta - \frac{\delta L}{2} e^{-Lt_{\eps}^{\delta}}\right) \times\left(\int_{\mathbb{R}^{d}}\sqrt{1+|x|^{2}}m_{\eps}^{\delta}(dx) + (m_{\eps}^{\delta})^2(\R^d) \right) \\
&\quad + 2\left(L_{G}(C + 2n_{\eps}^{\delta}(\R^d))e^{-Ls_{\eps}^{\delta}}\delta - \frac{\delta L}{2} e^{-Ls_{\eps}^{\delta}}\right) \times\left(\int_{\mathbb{R}^{d}}\sqrt{1+|x|^{2}}n_{\eps}^{\delta}(dx) + (n_{\eps}^{\delta})^2(\R^d)\right),
\end{align*}
where the first inequality holds true by Assumption \ref{A.ham} (i), and the second one by Assumption \ref{A.ham} (ii). Letting $\eps$ go to 0, one obtains
\begin{align*}
h \leq & ~ 2\left(L_{G}(C + 2m^{\delta}(\R^d))e^{-Lt^{\delta}}\delta - \frac{\delta L}{2} e^{-Lt^{\delta}}\right) \times\left(\int_{\mathbb{R}^{d}}\sqrt{1+|x|^{2}}m^{\delta}(dx) + (m^{\delta})^2(\R^d)\right) \\
 & + 2\left(L_{G}(C + 2n^{\delta}(\R^d))e^{-Ls^{\delta}}\delta - \frac{\delta L}{2} e^{-Ls^{\delta}}\right) \times\left(\int_{\mathbb{R}^{d}}\sqrt{1+|x|^{2}}n^{\delta}(dx) + (n_{\eps}^{\delta})^2(\R^d)\right).
\end{align*}
Finally, according to the second part of the proof of Lemma \ref{compact}, $m^{\delta}(\R^d)$ and $n^{\delta}(\R^d)$ have a common upper bound $c_{\delta}$, then choose  \(L \geq 2 L_G (C + 2c_{\delta})\), one gets the desired contradiction.

\end{proof}

\begin{remark}
	\label{remark:auxiliary_function_modification}
	A critical distinction between our comparison principle and that established in \cite{Erhan He 2025} lies in the construction of the auxiliary function $\vartheta$, which is essential for penalization and compactness arguments in the doubling variables method.
	
	In the setting of \cite{Erhan He 2025}, which deals with the space of probability measures $\mathcal{P}_2(\mathbb{R}^d)$, the total mass of any measure is fixed at 1. This allows for the use of an auxiliary function of the form $\vartheta(t, \mu) = 1 + \int_{\mathbb{R}^d} |x|^2 \mu(dx)$. The boundedness of the total mass ensures that the sub-level sets of this function are compact under the weak convergence topology, a property heavily relies upon in their proofs.
	
	However, this framework is insufficient for the space $\mathcal{M}_2(\mathbb{R}^d)$ considered in the present work, where elements are finite non-negative measures with unbounded total mass. The lack of an upper bound on $m(\mathbb{R}^d)$ creates two significant technical obstacles that render the auxiliary function from \cite{Erhan He 2025} inapplicable:
	\begin{itemize}
		\item \textbf{Loss of Compactness:} The sub-level sets $\{\vartheta(\theta) \leq c\}$ are no longer guaranteed to be compact. A sequence of measures $\{m_n\}$ with uniformly bounded second moments but diverging total masses ($m_n(\mathbb{R}^d) \to \infty$) would satisfy $\vartheta(\theta_n) \leq c$ yet fail to have a weakly convergent subsequence in $\mathcal{M}_2(\mathbb{R}^d)$, as the limit would not be a finite measure. This breaks a fundamental step in the proof of the comparison principle.
		\item \textbf{Unbounded Integrals:} The estimates involving the Hamiltonian $G$ and the linear growth assumptions on the solutions inherently involve integrals with respect to the measure $m$. When $m(\mathbb{R}^d)$ is unbounded, these integrals can become uncontrolled, preventing the derivation of the necessary contradiction in the argument.
	\end{itemize}
	
	To overcome these challenges intrinsic to $\mathcal{M}_2(\mathbb{R}^d)$, we introduce a modified auxiliary function defined in \eqref{aux}:
	\[
	\vartheta(t, m) = e^{-L t} \left( \int_{\mathbb{R}^d} \sqrt{1 + |x|^2}  m(dx) + m^2(\mathbb{R}^d) \right).
	\]
	The key innovation is the inclusion of the term $m^2(\mathbb{R}^d)$. This term grows quadratically with the total mass, effectively penalizing sequences with unbounded mass and ensuring that the sub-level sets $\{\vartheta(\theta) \leq c_1 + c_2 m(\mathbb{R}^d)\}$ are indeed compact, as proven in Lemma \ref{compact}. Furthermore, this modification allows us to maintain control over the various integral terms throughout the doubling variable argument, ultimately enabling us to establish the comparison principle in the more general space $\mathcal{M}_2(\mathbb{R}^d)$. This adaptation is a crucial technical contribution of our work, extending the applicability of the viscosity solution theory to PDEs associated with branching processes where the population size is dynamic and uncontrolled.
	\end{remark}

\section{Application : a branching control problem}\label{section4}
\label{sec:Formulation}

	Let us first introduce a controlled McKean--Vlasov branching diffusion process as solution of a branching stochastic differential equation (SDE), 
	and then derive some a priori estimates for the associated controlled processes.
	Throughout the paper, we fix 
	{ a constant $C_0> 0$,}
	together with a closed convex subset $A \subset \R^n$, which serves as the control space, i.e. the control processes take values in $A$.

\subsection{The controlled McKean--Vlasov branching diffusions} 
Let us first describe the McKean--Vlasov branching diffusion process in a heuristic way.  
	The coefficients of the controlled branching diffusion are given by
$$
	\big( b,\sigma, \gamma, (p_{\ell})_{\ell \in \N} \big)
	:
	[0,T] \x\R^d\x\Mc_2(\R^d) \x A \longrightarrow \R^d\x\R^{d\x d} \x  [0, \gammab] \x [0,1]^{\N}, 
$$
where $\gammab > 0$ is a fixed constant.
Namely, $b$ and $\sigma$ are the drift and diffusion coefficient for the movement of each particle, $\gamma$ is the death rate, 
and $(p_{\ell})_{\ell \in \N}$ is the probability mass function of the progeny distribution. 
In particular, it holds that $p_{\ell}(\cdot) \in [0,1]$ for each $\ell \in \N$, and $\sum_{\ell \in \N} p_{\ell}(\cdot) = 1$.
Let us also define a partition $(I_{\ell}(\cdot))_{\ell \in \N}$ of $[0,1)$ by
$$
	I_{\ell} (\cdot)
	~:=~
	\Big[
		\sum_{i=0}^{\ell -1}p_i (\cdot), ~\sum_{i=0}^{\ell} p_i (\cdot) 
	\Big),
	~~\mbox{for each}~
	\ell \in \N.
$$

Let $(\Om, \Fc, \F, \P)$ be a filtered probability space, with filtration $\F = (\Fc_t)_{t \ge 0}$, 
	equipped with a family of $d$-dimensional Brownian motion $(\Wk)_{k\in\K}$ and Poisson random measures $(\Qk(ds,dz))_{k\in\K}$ on $[0,T] \x [0, \gammab] \x [0,1]$ with Lebesgue intensity measure $ds\x dz$.
	The random elements $(\Wk)_{k\in\K}$ and $(\Qk(ds,dz))_{k\in\K}$ are mutually independent.
	Moreover, $\Fc_0$ is rich enough to support a $E$-valued random variable $\xi$ of any distribution.

	\vspace{0.5em}

	Let $K_s$ denotes the collection of all labels of particles alive at time $s \in [t,T]$. We here introduce the dynamic of each alive particle $k\in K_s$, which is given by the controlled SDE:
\begin{equation}
	d\Xk_s ~=~ b(s, \Xk_s,\mu_s, \alpha^k_s)ds+\sigma(s, \Xk_s,\mu_s,\alpha^k_s)d\Wk_s,
\end{equation}
with $\mu_s\in\Mc(\R^d)$ being defined by
	\begin{equation}
	\<\varphi,\mu_s\> 
	~:= ~
	\E\Big[
		\sum_{k\in K_s}\varphi(\Xk_s)	
	\Big],
	~~\mbox{for all}~
	\varphi\in C_b(\R^d),
	\end{equation}
as well as a notion of a closed-loop control, that is, the control depends on the position of the alive particles  by
\begin{equation*} 
	\alpha^k_s := \alpha(s,X^k_s),
	~~s \in [t,T], ~k \in \K,
\end{equation*}
with some measurable function $\alpha : [0,T] \x \R^d \longrightarrow A$.

\begin{definition}[Admissible controls]
	{
	Let $C_0 > 0$ be the fixed constant given at the beginning of Section \ref{sec:Formulation},}
	we define $\Ac$ as the space of all control processes $(\alpha^k)_{k \in \K}$ such that
	$\alpha^k_s = \alpha(s, \Xk_s)$ for some $\alpha:[0,T]\x\R^d \longrightarrow A$,
	where $\alpha(\cdot)$ satisfies the following conditions:
	\begin{align*}
		\big|\alpha(s,x) - \alpha(s,x')\big|
		~\leq~
		C_0 |x-x'|,
		~~\mbox{and}~~
		\big|\alpha(s,x)\big|
		~\leq~
		C_0 \big(1 + | x |\big),
	\end{align*} 
	for all $(s,x,x')\in[t,T]\x\R^d\x\R^d $.
\end{definition}
\vspace{0.5em}
Now for any initial time $t \in [0,T]$, and $E$-valued $\Fc_0$-random variable $\xi$ such that
	\begin{align}\label{assumption xi}
	\E[\<\xi, 1\>]
	~<~
	\infty,
	~~
	\text{and}~~
	\E[\<\xi, |\cdot|^2\>]
	~<~
	\infty,
	\end{align}	
	we consider a controlled branching diffusion process with an initial state $\xi$.
	It is represented as a $E$-valued process $(Z_s)_{s \in [t,T]}$ given by
	\begin{equation} \label{eq:Xk2Z}
		Z_s ~:=~ \sum_{k\in K_s}\delta_{(s,\Xk_s)}, ~~s \in [t,T].
	\end{equation}
	In particular, the initial condition of the branching process is given as
	$$
		Z_t ~=~ \xi ~=~ \sum_{k \in K_t} \delta_{(k, X^k_t)}.
	$$

	The branching process induces a flow of marginal measures $(\mu_s)_{s\in[t,T]}$, with $\mu_s\in\Mc(\R^d)$ being defined by
	\begin{equation}\label{eq:def_mu_t}
	\<\varphi,\mu_s\> 
	~:= ~
	\E\Big[
		\sum_{k\in K_s}\varphi(\Xk_s)	
	\Big],
	~~\mbox{for all}~
	\varphi\in C_b(\R^d).
	\end{equation}
Denote by $\Lc$ the infinitesimal generator of the diffusion with coefficient $(b,\sigma)$, \ie, for all $(s, x ,m)\in[t,T]\x\R^d\x\Mc(\R^d)$ and $\varphi \in C^2_b(\R^d,\R)$,
\begin{equation} \label{eq:def_Lc}
	\Lc \varphi (s, x, m, a)
	~:=~
	\frac{1}{2}\Tr\left(\sigma\sigma^{\top} (s, x,m, a)\nabla^2_x \varphi( x)\right) + b(s,x,m, a)\cdot \nabla_x \varphi(x),
\end{equation}
where we denote by $\nabla_x, \nabla^2_x$ the gradient and Hessian operators acting on the space variables respectively.
Then the McKean--Vlasov branching diffusion,
	with initial condition $\xi$ and control $\alpha \in \Ac$,
	can be characterized as the solution to the following SDE: for all $f:=(f^k)_{k\in\K}\in C^2_b (\K\x\R^d, \R),$ $s\in[t,T],$
\begin{multline} \label{eq:SDE_MKVB}
\< Z_s, f \>
= 
\<\xi ,f\>
+
\int_t^s \sum_{k\in K_u}\Lc f^k(u, \Xk_u, \mu_u,\alpha^k_u)du
+
\int_t^s \sum_{k\in K_u}
\nabla_x f^k(\Xk_u)\sigma(u,\Xk_u,\mu_u,\alpha^k_u)
d\Wk_u \\		
+\int_{(t,s]\x [0,\gammab] \x [0,1]}
\sum_{k\in K_{u-}}\sum_{ \ell \geq 0}
\left(\sum_{i=1}^{ \ell} 
f^{ki}-f^k\right)
(\Xk_u)\1_{[0,\gamma(u,\Xk_u,\mu_u,\alpha^k_u)]\x I_{\ell}(u,\Xk_u,\mu_u,\alpha^k_u)}(z)\Qk(du,dz),
\end{multline}
where we recall that the interaction term $\mu_u$ is defined in \eqref{eq:def_mu_t}.

\begin{definition}\label{Solution Def}
	Let $t \in [0,T]$, and $\xi$ be a $E$-valued {$\Fc_t$-random variable}, and $(\alpha^k)_{k \in \K} \in \Ac$.
	A solution to SDE \eqref{eq:SDE_MKVB} with initial data $(t, \xi)$ and control $(\alpha^k)_{k \in \K}$ is a $E$-valued $\F$-adapted càdlàg process $Z := (Z_s)_{s\in[t,T]}$ such that \eqref{eq:SDE_MKVB} holds true for all $s\in[t,T]$. 
\end{definition}

The coefficients functions will be assumed to satisfy the following conditions.

\begin{assumption}\label{A.1}
	\noindent $\mathrm{(i)}$ The coefficient functions $b,\sigma,\gamma$ are Lipschitz in $(x,m,a)$ in the sense that, there exists a constant $C>0$ such that
	\begin{align*}
		\big|(b,\sigma,\gamma)(s, x ,m, a) - (b,\sigma,\gamma)(s', x',m', a') \big| 
		~\leq~
		C
		\big(|s-s'| + |x-x'| + \rho_F(m, m')  + |a-a'|\big),
	\end{align*}
	for all $(s,s',x, x', m, m', a, a') \in [t,T] \x [t,T] \x \R^d \x \R^d \x \Mc_2(\R^d) \x \Mc_2(\R^d) \x A \x A$.

	\vspace{0.5em}

	\noindent $\mathrm{(ii)}$ The coefficient function $b$  and $\sigma$ are bounded both uniformly in $s\in[t,T]$.

	\vspace{0.5em}	
	\noindent $\mathrm{(iii)}$ The function $\sum_{ \ell \ge 0} \ell p_{\ell}$ is uniformly bounded by $M_1>0$. And there exist positive constants $(C_{\ell})_{\ell \ge 0}$ such that $M := \sum_{\ell \in\N} \ell C_{\ell} < \infty$ and 
	\begin{align*}
		\big| p_\ell (s, x,m, a)-p_\ell (s, x',m',a') \big| 
		~\leq~ 
		C_{\ell} \big(|s-s'| + |x-x'| + \rho_F(m, m') + |a-a'|\big),
	\end{align*}
	for all $(s, s', x, x', m, m',a,a') \in [t,T] \x [t,T] \x \R^d \x \R^d \x \Mc_2(\R^d) \x \Mc_2(\R^d) \x A \x A$ and $\ell \in \N$.

	\vspace{0.5em}

	\noindent $\mathrm{(iv)}$ { For any $(s, m,a) \in [t,T] \x \Mc_2(\R^d) \x A$, $x \mapsto \sigma^{\top}(s,x,m,a)\sigma(s,x,m,a) \in H^{\lambda+d/2+1}(\mathbb{R}^{d})$, $x \mapsto b(s,x,m,a) \in H^{\lambda}(\mathbb{R}^{d})$, $x \mapsto \sum_{\ell\geq0}(\ell - 1)p_\ell\big(s,x,m,a\big) \in H^{\lambda}(\mathbb{R}^{d})$, where $\lambda$ is given by \eqref{eq:lambda_def}.
	}

	\vspace{0.5em}

	\noindent $\mathrm{(v)}$ The diffusion matrix is uniformly non-degenerate. Namely, there exists a constant $\delta > 0$ such that for all $(s, x, m, a) \in [t, T] \times \mathbb{R}^d \times \mathcal{M}_2(\mathbb{R}^d) \times A$ and for all $\xi \in \mathbb{R}^d$:
$$\xi^\top (\sigma\sigma^\top(s, x, m, a)) \xi \ge \delta |\xi|^2$$
	
\end{assumption}

\subsection{A priori estimations}
We next prove that the controlled branching diffusion SDE has a unique solution and then provide some a priori estimates.
	Let us introduce a function $h:\Pc_1 (E) \longrightarrow \Mc(\R^d)$ by, for all $\nu \in \Pc_1 (E)$,
	\begin{equation}\label{definition h}
		\big\langle h(\nu), \varphi \big \rangle 
		:= \int_E \mathrm{Id}_{\varphi}(e)\nu(de),
		~~\mbox{for all}~\varphi\in C_b(\R^d),
	\end{equation}
	so that the marginal measure $\mu_s$ in \eqref{eq:def_mu_t} can be equivalently defined as $\mu_s :=h\big(\Lc(Z_s)\big)$.

\begin{lemma}\label{lemma continuity}
	The function $h:\Pc_1(E)\longrightarrow \Mc(\R^d)$ as defined in \eqref{definition h} is Lipschitz continuous in the sense that, for some constant $C> 0$,
	\begin{align*}
		\rho_F\big(h(\mu_1), h(\mu_2)\big)
		~\leq~
		C\Wc_1(\mu_1, \mu_2),~~\text{for all }\mu_1, \mu_2\in\Pc(E).
	\end{align*}
\end{lemma}

\begin{proof}
    Let us consider two $E$-valued random variables $Z_1,Z_2$ such that $Z_1 = \sum_{k\in K_1}\delta_{(k,X^k)}$ and $Z_2 = \sum_{k\in K_2}\delta_{(k,Y^k)}$ with $\mu_1 = \mathcal{P}\circ Z_1^{-1}$ and $\mu_2 = \mathcal{P}\circ Z_2^{-1}$.
\[
F_{n}\left(h\left(\mu_{1}\right)\right)
= (2\pi)^{-\frac{d}{2}}\int e^{in \cdot x}  dh(\mu_{\mathrm{1}})(x) =
(2\pi)^{-\frac{d}{2}}\mathbb{E_{(\mu_{\mathrm{1}})}}\left[\sum_{k\in K_{1}}e^{in\cdot X^{k}}\right].
\]
Consider $\gamma\in\Lambda (\mu_1,\mu_2)\subset\mathcal{P}(E\times E)$, and apply Jensen's inequality, we get
\[
\left|F_{n}\left(h\left(\mu_{1}\right)-h\left(\mu_{2}\right)\right)\right|
\leq
(2\pi)^{-\frac{d}{2}}\mathbb{E_\gamma}\left|\sum_{k\in K_{1}}e^{in\cdot X^{k}} - \sum_{k\in K_{2}}e^{in\cdot Y^{k}}\right|.
\]
Using triangle inequality and $\left|e^{i na}-e^{inb}\right| = 2\left|sin\frac{na-nb}{2}\right| \leq\min (2,|n||a-b|)$,

\begin{align*}
\left|F_{n}\left(h\left(\mu_{1}\right)-h\left(\mu_{2}\right)\right)\right| 
& \leq  (2\pi)^{-\frac{d}{2}}\mathbb{E_\gamma}\left| \sum_{k\in K_1\cap K_2}\left|e^{i n \cdot X^{k}}-e^{i n \cdot Y^{k}}\right| + \sum_{k \in K_{1} \backslash K_{2}} 1 + \sum_{k \in K_{2} \backslash K_{1}} 1 \right|\\
 & \leq (2\pi)^{-\frac{d}{2}}\mathbb{E_\gamma}\left| |n| \sum_{k \in K_{1} \cap K_{2}}\left(\left|X^{k}-Y^{k}\right| \wedge 1\right) + \#\left(K_{1} \triangle K_{2}\right) \right|.
\end{align*}
Therefore, taking infimum over all $\gamma\in\Lambda (\mu_1,\mu_2)$ , \\$\left|F_{n}\left(h\left(\mu_{1}\right)-h\left(\mu_{2}\right)\right)\right| \leq (2\pi)^{-\frac{d}{2}}|n+1| \mathbb{E_\gamma}\left[d_{E}\left(Z_{1}, Z_{2}\right)\right]= (2\pi)^{-\frac{d}{2}}|n+1|\mathcal{W}_{1}\left(\mu_{1}, \mu_{2}\right)$. Hence:

\[
\rho_{F}^{2}\left(h\left(\mu_{1}\right), h\left(\mu_{2}\right)\right)
\leq \mathcal{W}_{1}^{2}\left(\mu_{1}, \mu_{2}\right) \int_{\R^d} \frac{(2\pi)^{-d}|n+1|^{2}}{\left(1+|n|^{2}\right)^{\lambda}} dn 
= C \mathcal{W}_{1}^{2}\left(\mu_{1}, \mu_{2}\right).
\]

\end{proof}

\begin{proposition}
	Let Assumption \ref{A.1} hold true, $(t,\xi)$ be the initial condition with $t \in [0,T]$ and $\xi$ being a $\Fc_t$-random variable such that $\E[\langle \xi, 1\rangle] < \infty$, 
	and $(\alpha^k)_{k \in \K} \in \Ac$.
	Then there exists a unique solution of the controlled branching SDE \eqref{eq:SDE_MKVB} in the sense of Definition \ref{Solution Def}.
\end{proposition}

\begin{proof}
	We will apply the well-posedness results of the McKean-Vlasov branching SDE in Claisse, Kang and Tan \cite[Theorem 2.3]{Claisse Kang Tan 2024}.
	First, by the Lipschitz continuity of $h:\Pc(E)\longrightarrow\Mc(\R^d)$ in Lemma \ref{lemma continuity} as well as the Lipschitz property of the closed-loop control,
	it follows that $(x,m)  \longmapsto \big(b, \sigma, \gamma, p_{\ell} \big)(t, x, h(m), \alpha(x))$ is also Lipschitz continuous.
	Similarly, one can check the linear growth condition of $x \longmapsto \big(b, \sigma)(t,x, h(m), \alpha(x))$ 
	as required in the \cite[Theorem 2.3]{Claisse Kang Tan 2024}.
\end{proof}

\subsection{The control problem}\label{sec: control problem}

We now introduce a closed-loop McKean-Vlasov branching process control problem, inspired by the classical closed-loop McKean-Vlasov control problem in \cite{Pham Wei 2018}. We first introduce the cost functional and the value function of our control problem as follow.

\subsubsection{The value function}

	Let functions $L: [0,T] \times \mathbb{R}^d \times \Mc_2(\R^d) \times A \rightarrow \R$, $g: \mathbb{R}^d \times \Mc_2(\R^d) \rightarrow \mathbb{R}$ be
	the cost function. 
	We assume the following conditions.

	\begin{assumption}\label{A.3}
		\noindent $\mathrm{(i)}$ For all $(s,x,m,a)\in [0,T]\x\Mc_2(\R^d)\x A$, there exist constants $C_L, C_g$ such that
		\begin{align*}
			\big|L(s, x ,m, a) \big| 
			~\leq~
			C_L,\quad
			\big|g( x ,m) \big| 
			~\leq~
			C_g.
		\end{align*}
		\noindent $\mathrm{(ii)}$  $L$ is Lipschitz in $(t,x,m,a)$ in the sense that, there exists a constant $L_L>0$ such that
		\begin{align*}
			\big|L(s, x ,m, a) - L(s', x',m', a') \big| 
			~\leq~
			L_L
			\big( |s-s'|+|x-x'| + \rho_F(m, m')  + |a-a'|\big),
		\end{align*}
		for all $(s, s',x, x', m, m', a, a') \in [0,T] \x [0,T] \x \R^d \x \R^d \x \Mc_2(\R^d) \x \Mc_2(\R^d) \x A \x A$.\\
		\noindent $\mathrm{(iii)}$ For any $(a,s,m)$, the map $x \mapsto L(s,x,m,a)$ lies in $H^{\lambda}(\mathbb{R}^{d})$.
	\end{assumption}

	Let $t \in [0,T]$, $\xi$ be a $E$-valued $\Fc_t$-random variable and $\alpha\in\Ac$, 
	we denote by 
	$$
		Z^{t,\xi,\alpha}_s := \sum_{k\in K^{t,\xi,\alpha}_s}\delta_{(k,X^{k,t,\xi,\alpha}_s)}, s\in [t,T],
	$$
	the unique solution of \eqref{eq:SDE_MKVB}. 
	The cost function associated with the McKean--Vlasov branching diffusion process is given by
\begin{align}\label{eq:cost}
	J(t,\xi,\alpha)
	~:=&~
	\E\bigg[\int_t^T 
		\sum_{k\in K^{t,\xi,\alpha}_s}L\big(s,X^{k,t,\xi,\alpha}_s,\mu^{t,\xi, \alpha}_s,\alpha_s(X^{k,t,\xi,\alpha}_s)\big)	
		ds
	+
	\sum_{k\in K^{t,\xi,\alpha}_T}g(X^{k,t,\xi,\alpha}_T,\mu^{t,\xi, \alpha}_T)
	\bigg] \nonumber \\
	~=&~
	\int_t^T \big\<L\big(s, \cdot, \mu^{t,\xi, \alpha}_s, \alpha_s(\cdot)\big),\mu^{t,\xi, \alpha}_s\>ds
	+
	\big\<g(\cdot, \mu^{t,\xi, \alpha}_T), \mu^{t,\xi, \alpha}_T\big>,
\end{align}
where $\mu^{t,\xi, \alpha}_s$ is defined for any $s\in[t,T]$ and for any $\varphi\in C_b(\R^d)$,
\begin{align}\label{eq:mean measure xi}
	\<\varphi,\mu^{t,\xi, \alpha}_s\>
	:=
	\mathbb{E}\bigg[\sum_{k \in K^{t,\xi,\alpha}_s} \varphi(X^{k,t,\xi,\alpha}_s)\bigg].
\end{align}

\begin{proposition}[{\cite[Proposition 3.2]{Claisse Kang Lan Tan 2025}}]
	Let Assumptions \ref{A.1} and \ref{A.3} hold true, then the cost functional $J$ in \eqref{eq:cost} is well-defined and finite. 
\end{proposition}
\vspace{1em}	
Now for all $(t, \nu) \in [0,T] \x \Mc(\R^d)$, let us define
$$
	\Xi(t, \nu) 
	~:=~
	\big\{ \xi ~\mbox{is}~\Fc_t\mbox{-measurable}~E\mbox{-valued random variable s.t.}~ h(\Lc(\xi)) = \nu \big\}.
$$
Equivalently, $\Xi(t, \nu)$ is the set of all $\Fc_t$-measurable $E$-valued random variables $\xi = \sum_{k\in K} \delta_{X^k_t}$ such that $ \E \big[ \sum_{k\in K} \varphi(X^k_t) \big] = \langle \nu, \varphi \rangle$ for all $\varphi \in C_b(\R^d)$.
It is clear that $\Xi(t, \nu)$ is non-empty for all $\nu \in \Mc(\R^d)$.
Let us then introduce the value function of our  controlled branching processes problem:
\begin{equation}\label{eq:value function}
	v(t, \nu)
	~:=~
	\inf_{\alpha \in \Ac} ~\inf_{\xi \in \Xi(t, \nu)} J(t,\xi,\alpha),
	~~\mbox{for all}~(t, \nu) \in [0,T] \x \Mc_2(\R^d).
\end{equation}

In {\cite[Lemma 3.3]{Claisse Kang Lan Tan 2025}}, it is shown that $J(t, \xi_1, \alpha) = J(t, \xi_2, \alpha)$ for all $\xi_1, \xi_2 \in \Xi(t,\nu)$, so that one can replace the infimum over $\xi$ in \eqref{eq:value function}
by taking an arbitrary $\xi \in \Xi(t, \nu)$. Consequently, one can denote $\mu^{t,\nu, \alpha} := \mu^{t,\xi, \alpha}$ with an arbitrary $\xi \in \Xi(t, \nu)$, where $(\mu^{t,\xi, \alpha}_s)_{s \in [t,T]}$ is the flow of marginal measure defined in \eqref{eq:mean measure xi},
	so that 
	$$
		v(t, \nu) = \inf_{\alpha \in \Ac} J(t, \nu, \alpha),
	$$
	and
	\begin{align}\label{eq:reformulation value function}
		v(t, \nu)
		~=&~
		\inf_{\alpha\in\Ac} \bigg\{
		\int_t^T \big\<L\big(s,\cdot,\mu^{t,\nu,\alpha}_s,\alpha_s(\cdot)\big),\mu^{t,\nu,\alpha}_s\big\>ds
		+
		\big\<g(\cdot,\mu^{t,\nu,\alpha}_T),\mu^{t,\nu,\alpha}_T\big\>
		\bigg\}.
	\end{align}

\subsubsection{Dynamic programming principle}

We now provide the dynamic programming principle of the controlled branching processes problem.
Recall that $\mu^{t,\nu, \alpha} := \mu^{t,\xi, \alpha}$ with an arbitrary $\xi \in \Xi(t, \nu)$, where the latter is defined in \eqref{eq:mean measure xi}.

\begin{lemma}[{\cite[Lemma 3.4]{Claisse Kang Lan Tan 2025}}]\label{lemma flow property}
	Let Assumption \ref{A.1} hold true.
	Then we have the flow property for the measure $\mu^{t, \nu}$ in the sense that for all $t \le u \le s \le T$,
\begin{align*}
	\mu^{t,\nu,\alpha}_s
	~=&~
	\mu^{u, \mu^{t,\nu,\alpha}_u,\alpha}_s.
\end{align*}
\end{lemma}

\begin{theorem} \label{DPP}
	Let Assumptions \ref{A.1} and \ref{A.3} be true,
	and $v$ be the value function of the controlled branching process in \eqref{eq:reformulation value function}.
	Then, for all $(t, \nu) \in [0,T] \x \Mc_2(\R^d)$ and $s \in [t,T]$, one has
	\begin{equation}\label{eq:dpp}
		v(t, \nu)
		~=~
		\inf_{\alpha\in\Ac}\bigg\{
			\int_{t}^{s} \<L\big(u,\cdot,\mu^{t,\nu,\alpha}_u, \alpha_u(\cdot)\big),\mu^{t,\nu,\alpha}_u\>du
		+
		v(s, \mu^{t,\nu,\alpha}_s)
		\bigg\}.
	\end{equation}
\end{theorem}

{
	Notice that our control problem \eqref{eq:value function} differs slightly with that in \cite{Claisse Kang Lan Tan 2025} since our admissible control are required to be Lipschitz in the space variable with fixed Lipschitz constant $C_0 > 0$, in place of any measurable function in \cite{Claisse Kang Lan Tan 2025}.
	The proof of our dynamic programming result is almost the same as that in \cite[Theorem 3.2]{Claisse Kang Lan Tan 2025}, and is hence omitted.
}

\subsubsection{The Hamilton-Jacobi-Bellman equation and its comparison principle}
	
{
	Recall the fixed constant $C_0 > 0$ at the beginning of Section \ref{sec:Formulation}, 
}
	we denote by $\tilde{\mathcal{A}}$ the space of all measurable deterministic function $\tilde\alpha:\R^d\rightarrow A$ satisfying :
	\begin{itemize}
		\item \(\big|\tilde\alpha(x) - \tilde\alpha(x')\big|
		~\leq~
		C_0 \big(|x-x'|\big)\) for all $(x,x')\in\R^d\x\R^d$,
		\item \(\big|\tilde\alpha(x)\big|
		~\leq~
		C_0 \big(1 + | x |\big)\).
	\end{itemize}
	For all $\tilde \alpha \in \tilde{\mathcal{A}}$, we define $\Gc^{\tilde \alpha}_t v$ by, for any $\nu\in\Mc_2(\R^d)$ and $x\in\R^d$,
\begin{align*}
	\Gc^{\tilde\alpha}_t v(\nu)(x)
	~:=&~
	b\big(t,x,\nu,\tilde\alpha(x)\big)\cdot D_\mu v(t,\nu)(x) 
	+
	\frac{1}{2}\Tr\left(\sigma\sigma^{\top}\big(t,x,\nu,\tilde\alpha(x)\big)\partial_x D_\mu v(t,\nu)(x)\right)\\
	&~+
	\gamma\big(t,x,\nu,\tilde\alpha(x)\big)\sum_{\ell\geq0}(\ell - 1)p_\ell\big(t,x,\nu,\tilde\alpha(x)\big)\frac{\delta v}{\delta\mu}(t,\nu)(x).
\end{align*}

\begin{theorem}\label{hjbprop}
	Let Assumption \ref{A.1} and \ref{A.3} hold. 
	Then the value function $v$ defined in \eqref{eq:value function} is the unique viscosity solution of the following HJB equation:
	\begin{equation}\label{eq:HJB}
		\partial_t v(t,m) + \inf_{\tilde{\alpha}\in\tilde{\Ac}}\big\{
			\<L\big(t,\cdot,m, \tilde{\alpha} (\cdot)\big),m\>
			+
			\<\Gc^{\tilde{\alpha}}_tv(t,m)(\cdot),m\>
		\big\}
		=
		0, 
		~~~(t,m) \in [0,T)\times\Mc_2(\R^d),\\
	\end{equation}
	with terminal condition 
	$v(T,m) =	\<g(\cdot, m), m\> $.
\end{theorem}

\begin{remark}
	\label{remark:comparison_with_MFC}
	It is pertinent to compare our approach with the recent work of \cite{Erhan Hang 2025}, which establishes a comparison principle for a specific class of HJB equations arising from mean field control (MFC) problems. The methodology in \cite{Erhan Hang 2025} leverages a particle approximation scheme, constructing a sequence of finite-dimensional PDEs whose solutions converge to the viscosity solution of the limiting equation on the Wasserstein space. A significant advantage of this approach is that it operates under assumptions that are weaker than those required by the more general doubling variable arguments employed in \cite{Erhan Xin 2025}.

	The core of the method in \cite{Erhan Hang 2025} relies on the limit theory of the standard mean-field control problem, 
	which ensures that the mean-field control problem describes the limiting behavior of a large system of controlled interacting particles as the population size tends to infinity. 
	Although the limit theory for uncontrolled McKean-Vlasov branching diffusion process has already been established in \cite{Claisse Kang Tan 2024, CaoRenTan}, and it is natural to be extended to the controlled setting,
	such an extension seems not direct and trivial. 
	We would like to address it in our future project.

\end{remark}

\section{Proof of Theorem \ref{hjbprop}}\label{section5}

This chapter is devoted to the rigorous proof of Theorem \ref{hjbprop}, which characterizes the value function of the controlled branching McKean-Vlasov problem as the unique viscosity solution to the associated Hamilton-Jacobi-Bellman equation \eqref{eq:HJB}. The argument proceeds in three distinct stages. First, Section \ref{section5.1} establishes the requisite a priori estimates for the controlled branching diffusion , alongside fundamental stability results and Itô's formula for measure-valued processes. Building upon these tools, Section \ref{section5.2} deduces the critical regularity properties of the system; in particular, we establish a time-Hölder inequality for the marginal measures with a Hölder coefficient strictly independent of the chosen admissible control, as well as the continuity of the Hamiltonian. Finally, Section \ref{section5.3} synthesizes these components to verify the continuity of the value function and to prove that it satisfies the HJB equation in the viscosity sense, concluding with the verification of the comparison principle to guarantee uniqueness.

\subsection{Technical lemmas}\label{section5.1}
To lay the groundwork for the viscosity solution framework, this section derives essential a priori bounds and stability estimates for the controlled McKean-Vlasov branching diffusion. We begin by establishing uniform moment estimates for both the population size and the spatial distribution of the particles. These estimates are pivotal for ensuring the well-posedness of the cost functional and the finiteness of expectations taken over the entire temporal integration and summation of the branching processes. Subsequently, we recall Itô's formula for measure-valued processes and prove the robust stability of the branching dynamics under the bounded Lipschitz distance.

\subsubsection{Moment estimates for the controlled process}
\begin{lemma}\label{lemma integrability K}
	Let Assumption \ref{A.1} hold true,
	$(t,\xi)$ be the initial condition with $t \in [0,T]$ and $\xi$ being a $\Fc_t$-random variable, 
	and $(\alpha^k)_{k \in \K} \in \Ac$.
	Then with the constant $M_1$ in Assumption \ref{A.1}, one has first order estimate
	\begin{equation}\label{ineq:supK}
		\sup_{s\in[t,T]} \mu^{t,\xi,\alpha}_s(\R^d) 
		~\leq~
		\E\big[\sup_{s\in[t,T]} \# K^{\alpha}_s\big]
		~\leq~
		\E[\<\xi, 1\>] e^{\bar\gamma M_1 (T-t)},
	\end{equation}
	where $\mu^{t,\xi,\alpha}_s$ is the process as defined in \eqref{eq:mean measure xi}.
\end{lemma}

\begin{proof}
	By the definition of the marginal measure $\mu^{t,\xi,\alpha}_s$ given in \eqref{eq:mean measure xi}, evaluating it with the constant test function $\varphi \equiv 1$ yields the total mass of the measure at time $s$:
	\begin{align*}
		\mu^{t,\xi,\alpha}_s(\R^d) 
		~=~ \<\1, \mu^{t,\xi,\alpha}_s\> 
		~=~ \E\bigg[\sum_{k \in K^{\alpha}_s} 1\bigg] 
		~=~ \E\big[\# K^{\alpha}_s\big].
	\end{align*}
	Taking the supremum over $s \in [t, T]$ on both sides, and using the fundamental property that the supremum of an expectation is bounded above by the expectation of the supremum, we obtain:
	\begin{align*}
		\sup_{s\in[t,T]} \mu^{t,\xi,\alpha}_s(\R^d) 
		~=~ \sup_{s\in[t,T]} \E\big[\# K^{\alpha}_s\big] 
		~\leq~ \E\bigg[\sup_{s\in[t,T]} \# K^{\alpha}_s\bigg].
	\end{align*}
	This establishes the first inequality.
	The proof of the second inequality in \eqref{ineq:supK} follows exactly the same argument in Claisse \cite[Proposition 2.1]{Claisse 2018} which is omitted here. 
\end{proof}

\begin{lemma}\label{upperbound}
	Let Assumption \ref{A.1} hold,
	$\xi$ be a $\Fc_t$-random variable such that $\E[\<\xi,1\>] + \E[\<\xi, |\cdot|^2\>] < \infty$,
	and $(\alpha^k)_{k \in \K} \in \Ac$.
	Then
	\begin{align*}
		\E\bigg[\sup_{s\in [t,T]} \sum_{k\in K^{\alpha}_s}|X^{k,\alpha}_s|^2\bigg]
		~\leq~
		C.
	\end{align*}
\end{lemma}

\begin{proof}
	The proof of Lemma \ref{upperbound} follows exactly the same argument in \cite[Lemma~2.5]{Claisse Kang Lan Tan 2025}
\end{proof}

\subsubsection{Itô's formula for measure-valued processes}

	Recall that the linear functional derivatives and intrinsic derivatives for functionals defined on $\Mc_2(\R^d)$ are defined in Definition \ref{definition linear functional derivative}.

\begin{theorem}[{\cite[Theorem 2.14]{CaoRenTan}}]\label{proposition Ito formula}
	Let $F\in C^{1,2}([0,T] \x \Mc_2(\R^d))$,
	and $(\mu_t)_{t \in [0,T]}$ be given by \eqref{eq:def_mu_t} for a controlled branching diffusion process $Z$ with control $\alpha$.
	Then, for all $0 \le s \le t \le T$,
	\begin{align*}
		F(t, \mu_t) - F(s, \mu_s) &\nonumber  \\
		=
		\int_s^t
		\partial_t F(u, \mu_u) 
		+
		\Big\langle &
			b\big(u,\cdot,\mu_u,\alpha_u(\cdot)\big)\cdot D_\mu F(\mu_u, \cdot) \nonumber  \\
		&	+
			\frac{1}{2}\Tr\big(\sigma\sigma^{\top} \big(u, \cdot,\mu_u, \alpha_u(\cdot)\big)\partial_x D_\mu F(\mu_u,\cdot)\big) \nonumber \\
		&	~+
			\gamma\big(u,\cdot,\mu_u,\alpha_u(\cdot)\big)\sum_{\ell\geq0}(\ell - 1)p_\ell\big(u,\cdot,\mu_u,\alpha_u(\cdot)\big)\frac{\delta F}{\delta\mu}(\mu_u,\cdot)
			,
			\mu_u
		\Big\rangle
		du.
	\end{align*}
\end{theorem}

\subsubsection{Stability estimates of the process}

Here we establish the Lipschitz continuity of the branching marginal measures with respect to the initial measure in the bounded Lipschitz distance (as defined in Section \ref{T_metric}). This stability result is fundamental, demonstrating that two processes starting from similar initial configurations will remain close over time. This property is key to proving the continuity of the value function in the bounded Lipschitz distance.

\begin{lemma}\label{branching stability}
    Let Assumption \ref{A.1} hold true. For any $\alpha \in \Ac$, let $Z^{t,\xi,\alpha}_s$ and $Z^{t,\xi',\alpha}_s$ be respectively the unique solution to SDE with initial condition $\xi$ and $\xi'$ satisfying \eqref{assumption xi}. Then there exists a constant $C>0$  depend on \( (C_{b,\gamma,l,\sigma}, L, C_0, C_1, \bar\gamma,M) \) where \(C_1:=\max \left(\mathbb{E}\left[\# K_{0}\right], \mathbb{E}\left[\# K'_{0}\right]\right)\) such that
    \begin{align*}
        \mathbb{E}[d_E(Z^{t,\xi,\alpha}_s,Z^{t,\xi',\alpha}_s) ] \leq C \mathbb{E}[d_E(\xi,\xi') ].
    \end{align*}
\end{lemma}

\begin{proof}
    Denote $K^{\triangle}_t := K_t\triangle K'_t$ and $K^{\cap}_t = K_t\cap K'_t$ and recall that
    \begin{align*}
     d_E(Z^{t,\xi,\alpha}_s,Z^{t,\xi',\alpha}_s) = \# K^{\triangle,\alpha}_s + \sum_{k\in K^{\cap,\alpha}_s}  \big|X^{k,\alpha}_s-X'^{k,\alpha}_s\big|\wedge 1.
    \end{align*}
    Since $\alpha$ is a closed-loop control satisfies the Lipschitz conditions with a uniform Lipschitz constant not depend on $\alpha$. The result is a direct generalization of {\cite[Proposition A.1]{Claisse Kang Tan 2024}}.
\end{proof}

\begin{lemma}
	For any $m,m' \in \Mc_2(\R^d)$, $\alpha \in \Ac$, there exists a constant $C> 0$ (independent of $\alpha$) such that
	\[
		\mathbf{d}_{\mathrm{BL}}(\mu^{t,m,\alpha}_s,\mu^{t,m',\alpha}_s) \leq C\mathbf{d}_{\mathrm{BL}}(m,m').
	\]
\end{lemma}

	\begin{proof}
		First, by {\cite[Lemma 2.2]{Claisse Kang Lan Tan 2025}} and Lemma \ref{branching stability}, for any E-valued  random variable $\xi,\xi'$ satisfy $h(\mathcal{L}(Z)) = m, h(\mathcal{L}(Z')) = m'$, 
		\[
			\mathbf{d}_{\mathrm{BL}} (\mu^{t,m,\alpha}_s,\mu^{t,m',\alpha}_s) \leq \mathbb{E}[d_E(Z^{t,\xi,\alpha}_s,Z^{t,\xi',\alpha}_s) ] \leq C \mathbb{E}[d_E(\xi,\xi') ].
		\]
		So we left to show that for any $\varepsilon$ by some specific construction of $\xi,\xi'$, we have \(\mathbb{E}[d_E(\xi,\xi') ] \leq C\mathbf{d}_{\mathrm{BL}}(m,m') + \eps\).

		By definition of $\mathbf{d}_{\mathrm{BL}}$, it can be easily checked that
		\[
			\mathbf{d}_{\mathrm{BL}}(m, m^{\prime}) \geq
		C\min\left(m\left({\R}^{d}\right), m^{\prime}\left({\R}^{d}\right)\right)\cdot\mathcal{W}_{1}\left(\frac{m}{m\left({\R}^{d}\right)},\frac{m^{\prime}}{m^{\prime}\left({\R}^{d}\right)}\right)+\left|m\left({\R}^{d}\right)-m^{\prime}\left({\R}^{d}\right)\right|,
		\]
		where $\mathcal{W}_{1}$
		is the Wasserstein distance of probability measure with trancated Euclidian distance \\
		Since $m$ and $m'$ are finite positive measures, we shall assume that \(m(\R^d) = n + \lambda\), and \(m'(\R^d) = n' + \lambda'\), where $n$ and $n'$ are positive integers, and \(\lambda,\lambda' \in [0,1)\). Without loss of generality we may assume \(m(\R^d) \leq m(\R^d)' < N\) for some positive integer N\\
		Then, we can construct $N$ i.i.d. random pairs \((X_1,Y_1),...,(X_{N},Y_{N})\) with \( X_i \sim \frac{m}{m(\R^d)} \), \( Y_i \sim \frac{m'}{m'(\R^d)} \) and \[
		\mathbb{E}\left[\left|X_{i}-Y_{i}\right|\right] - \varepsilon \leq \mathcal{W}_{1}\left(\frac{m}{m\left(\mathbb{R}^{d}\right)}, \frac{m^{\prime}}{m^{\prime}\left(\mathbb{R}^{d}\right)}\right)
		\]
	 for all \(1 \leq i \leq N \). We also construct two Bernoulli random variables $\alpha$ and $\alpha'$ such that $\alpha = \mathbf{1}_{\{U \leq \lambda\}}$, $\alpha' = \mathbf{1}_{\{U \leq \lambda'\}}$, where $U \sim \text{Uniform}(0,1)$.

		Let us define
		\[
		\xi = \sum_{k = 1}^{n}\delta_{(k,X_k)} + \alpha \delta_{(n+1,X_{n+1})},
		\]
		\[
		\xi' = \sum_{k = 1}^{n}\delta_{(k,Y_k)} + \alpha' \delta_{(n+1,Y_{n+1})} + \sum_{k = n+2}^{n'+1}\delta_{(k,Y_k)}.
		\]
		Then by  \eqref{metric:E} we have, 
		\[
		d_{E}\left(\xi,\xi' \right)\leq\sum_{i=1}^{n}\left(\left|X_{i}-Y_{i}\right|\wedge 1\right) + min(\alpha, \alpha')\left(\left|X_{n+1}-Y_{n+1}\right|\wedge 1\right) + |\alpha' - \alpha| + n^{\prime} - n.
		\]
		Taking the expectation, it follows that
		\begin{align*}
		\mathbb{E}\left[d_{E}\left(\xi,\xi' \right)\right]
		& \leq\min\left(m\left(\mathbb{R}^{d}\right), m'\left(\mathbb{R}^{d}\right)\right)\cdot\mathcal{W}_{1}\left(\frac{m}{m\left(\mathbb{R}^{d}\right)},\frac{m'}{m'\left(\mathbb{R}^{d}\right)}\right)+\left|m\left(\mathbb{R}^{d}\right)-m'\left(\mathbb{R}^{d}\right)\right| + \varepsilon \\
		& \leq \mathbf{d}_{\mathrm{BL}}(m, m^{\prime})+ \varepsilon.
		\end{align*}
		\end{proof}

\subsection{Regularity properties of the marginal measures and the Hamiltonian}\label{section5.2}

With the foundational estimates in place, we now turn to the regularity properties of the branching marginal measures and the Hamiltonian, which are indispensable for demonstrating the existence of the viscosity solution to the HJB equation. The primary objective of this section is twofold. First, we strictly establish that the measure-valued process satisfies a $\frac{1}{2}$-Hölder inequality in time, yielding a Hölder coefficient that is independent of the admissible control $\alpha \in \mathcal{A}$. Second, we verify the continuity of the Hamiltonian $G$ with respect to its parameters, a necessary condition for the application of the comparison principle established in Chapter \ref{section3}.

\subsubsection{Time-Hölder continuity of the marginal measures}

\begin{lemma}\label{strong continuity}
	Let Assumption \ref{A.1} hold true, $\xi$ be a $\Fc_t$-random variable such that $\E[\<\xi,1\>] < \infty$. The measure valued process $\mu^{t,m,\alpha}_s$ as defined in \eqref{eq:mean measure xi} and \cite[Lemma 3.3]{Claisse Kang Lan Tan 2025} is $\frac{1}{2}$-Hölder continuous in time in the sense that there exists a constant $C > 0$, such that for any $r, s \in [t, T]$ with $r \leq s$, the following estimate holds:
	\begin{align*}
		\rho_F\big(\mu^{t,m,\alpha}_s, \mu^{t,m,\alpha}_r\big)
		~\leq~
		C \sqrt{\left|s-r\right|},
	\end{align*}
	where the constant C is independent of the choice of the admissible control $\alpha \in \mathcal{A}$.
\end{lemma}

\begin{proof}
	Recall that from Lemma \ref{lemma continuity}, we have that for any $\xi \in E$, \(s \geq r\),
	\begin{align*}
		\rho_F\big(\mu^{t,m,\alpha}_s, \mu^{t,m,\alpha}_r\big)
		~\leq~
		C \E\big[\sum_{k\in K_s\cap K_r}|X_s^k - X_r^k|\wedge1 + \# (K_s\triangle K_r)\big]
		~=~
		C \E\big[d_E(Z^{t,\xi,\alpha}_s, Z^{t,\xi,\alpha}_r)\big].
	\end{align*}
	\textbf{Part 1.}We estimate \(\# (K_s\triangle K_r)\)
Let us denote the number of alive particles at time t by $N_t$, and let $D(r,s)$ denote the number of particles dead in $[r,s]$, then we have \(\# (K_s\triangle K_r) = |N_s - N_r| + 2D(r,s)\). Consider SDE \eqref{eq:SDE_MKVB}, taking $f(x) = 1$, we get
	\begin{align*}
		N_s = \langle Z_s, f \rangle = N_r + \int_{(r,s] \times [0,\bar{\gamma}] \times [0,1]} \sum_{k \in K_{u-}} \sum_{\ell \geq 0} (\ell - 1) \1_{[0,\gamma] \times I_\ell}(z) Q^k(du, dz),
	\end{align*}
	\[
\mathbb{E}\left[\left|N_{s}-N_{r}\right|\right] \leq \mathbb{E}\left[\int_{r}^{s} \sum_{k \in K_{v-}} \gamma\left(v, X_{v}^{k}, \mu_{v}, \alpha_{v}^{k}\right) \sum_{\ell \geq 0}|\ell-1| p_{\ell}\left(v, X_{v}^{k}, \mu_{v}, \alpha_{v}^{k}\right)  dv\right].
\]
Since $\sum_{\ell} |\ell-1| p_{\ell}$ is bounded, let the bound be $M_1$, and $\gamma \leq \bar{\gamma}$, we have:

\[
\mathbb{E}\left[\left|N_s - N_{r}\right|\right] 
\leq \left(\bar{\gamma} M_1\right) (s - r) 
\mathbb{E}\left[ \sup_{v \in [r, s]} N_v \right].
\]
By a prior estimate, $\mathbb{E}\left[ \sup_{v} N_v \right] \leq C$ 
(where the constant $C$ does not depend on the control), we have:

\[
\mathbb{E}\left[\left|N_s - N_{r}\right|\right] \leq C_1 (s - r).
\]
As for the death part $D(r,s)$, we have
\[
\mathbb{E}\left[D(r,s)\right] = 
\mathbb{E}\left[\int_{r}^{s} \sum_{k \in K_{v}} \gamma\left(v, X_{v}^{k}, \mu_{v}, \alpha_{v}^{k}\right)  dv\right] 
\leq \bar{\gamma} \int_{r}^{s} \mathbb{E}\left[N_{v}\right]  dv 
\leq \bar{\gamma} C_2\left(s-r\right).
\]
Thus we have 
\[
\mathbb{E}\left[\# (K_s\triangle K_r)\right] \leq C (s - r).
\]
\textbf{Part 2.}Next, consider the change in position of particles with common labels, i.e.,
\[
\mathbb{E}\left[\sum_{k \in K_{r} \cap K_s} \left|X_s^k - X_{r}^k\right| \wedge 1\right].
\]
For a particle $k$ that survives at both $r$ and $s$, it does not undergo branching in the interval $[r, s]$. Therefore, within $[r, s]$, the evolution of particle $k$ follows a diffusion process without branching.\\
Hence, for each such particle $k$, its position satisfies:
\[
X_s^k = X_{r}^k + \int_{r}^{s} b\left(u, X_u^k, \mu_u, \alpha_u^k\right) du + \int_{r}^{s} \sigma\left(u, X_u^k, \mu_u, \alpha_u^k\right) dW_u^k.
\]
Therefore,
\[
\left|X_s^k - X_{r}^k\right| \wedge 1 \leq \left|\int_{r}^{s} b\left(u, X_u^k, \mu_u, \alpha_u^k\right) du\right| + \left|\int_{r}^{s} \sigma\left(u, X_u^k, \mu_u, \alpha_u^k\right) dW_u^k\right|,
\]
and 
\begin{align*}
	&\mathbb{E}\left[\sum_{k \in K_{r} \cap K_s} \left|X_s^k - X_{r}^k\right| \wedge 1\right]  \\
	\leq  &\mathbb{E}\left[\sum_{k \in K_{r} \cap K_s}\left|\int_{r}^{s} b\left(u, X_u^k, \mu_u, \alpha_u^k\right) du\right| + \sum_{k \in K_{r} \cap K_s}\left|\int_{r}^{s} \sigma\left(u, X_u^k, \mu_u, \alpha_u^k\right) dW_u^k\right|\right] \\
	\leq &\mathbb{E}\left[\sum_{k \in K_{r} \cap K_s}\int_{r}^{s} \left|b\left(u, X_u^k, \mu_u, \alpha_u^k\right)\right| du + \sum_{k \in K_{r} \cap K_s}\left|\int_{r}^{s} \sigma\left(u, X_u^k, \mu_u, \alpha_u^k\right) dW_u^k\right|\right].
\end{align*}
For the drift part, by the linear growth condition in our assumptions:
\begin{align*}
	\mathbb{E}\left[\sum_{k \in K_{r} \cap K_s}\int_{r}^{s} \left|b\left(u, X_u^k, \mu_u, \alpha_u^k\right)\right| du\right] 
\leq \mathbb{E}\left[\sum_{k \in K_{r} \cap K_s}\int_{r}^{s}\left(1 + \left|X_{u}^{k}\right| + \mu_u(\mathbb{R}^d) + \left|\alpha_{u}^{k}\right|\right)du\right]. 
\end{align*}
By the control condition \(\left|\alpha_{u}^{k}\right| \leq C\left(1 + \left|X_{u}^{k}\right| + \mu_{u}\left(\mathbb{R}^{d}\right)\right)\), together with Lemma \ref{lemma integrability K} we have:
\begin{align*}
	&\mathbb{E}\left[\sum_{k \in K_{r} \cap K_s}\int_{r}^{s} \left|b\left(u, X_u^k, \mu_u, \alpha_u^k\right)\right| du\right] 
\leq C \mathbb{E}\left[\int_{r}^{s}\sum_{k \in K_{r} \cap K_s}\left(1 + \left|X_{u}^{k}\right| + \mu_{u}\left(\mathbb{R}^{d}\right)\right)du\right] \\
\leq& C(s-r)\mathbb{E}\left[\sum_{k \in K_{r} \cap K_s}1\right] + C(s-r)\E\bigg[\sup_{u\in [r,s]} \sum_{k\in K_u}|X^{k}_u|\bigg],
\end{align*}
which is bounded by $C(s-r)$ by Lemma \ref{upperbound}.\\
And since the set $\K$ is countable, we shall prove that the diffusion term has a similar bound:\\
Let \(\Fc^k := \sigma(K_r,K_s)\),
\begin{align*}
	&\mathbb{E}\left[\sum_{k \in K_{r} \cap K_s}\left|\int_{r}^{s} \sigma\left(u, X_u^k, \mu_u, \alpha_u^k\right) dW_u^k\right| \right] 
	= \mathbb{E}\left[\sum_{k \in \K}\mathbf{1}_{\{k \in K_{r} \cap K_s\}}\left|\int_{r}^{s} \sigma\left(u, X_u^k, \mu_u, \alpha_u^k\right) dW_u^k\right| \right]\\
	= &\sum_{k \in \K}\mathbb{E}\left[\left|\int_{r}^{s} \mathbf{1}_{\{k \in K_{r} \cap K_s\}}\sigma\left(u, X_u^k, \mu_u, \alpha_u^k\right) dW_u^k\right| \right].
\end{align*}
Applying Burkholder-Davis-Gundy inequality, and denote the upper bound of $\sigma$ by $C_\sigma$ we get
\begin{align*}
	&\sum_{k \in \K}\mathbb{E}\left[\left|\int_{r}^{s} \mathbf{1}_{\{k \in K_{r} \cap K_s\}}\sigma\left(u, X_u^k, \mu_u, \alpha_u^k\right) dW_u^k\right| \right]\\
	\leq &\sum_{k \in \K}\mathbb{E}\bigg[\left(\int_{r}^{s} \mathbf{1}_{\{k \in K_{r} \cap K_s\}}C_{\sigma}^2 du^k\right)^{1/2} \bigg]
	 = C_\sigma \sqrt{s - r} \mathbb{E}\left[\sum_{k \in \K} \mathbf{1}_{\{k \in K_{r} \cap K_s\}}\right]\\
	\leq &C_\sigma \sqrt{s - r}\mathbb{E}\left[ \sup_{v \in [r, s]} N_v \right]
	\leq C \sqrt{s - r}.
\end{align*}
In summary, we have:
\begin{align*}
\mathbb{E}\left[d_{E}\left(Z_{s}, Z_r\right)\right]
&= \mathbb{E}\left[\sum_{k \in K_{r} \cap K_{s}} \left(\left|X_{s}^{k}-X_{r}^{k}\right| \wedge 1\right)\right] 
+ \mathbb{E}\left[\#\left(K_{s} \triangle K_{r}\right)\right]\\
&\leq C \sqrt{s-r} + C\left|s-r\right| + C\left|s-r\right|\\
&\leq C \sqrt{s-r}.
\end{align*}
Here, we used the inequality \( \left|s-r\right| \leq \sqrt{T} \sqrt{\left|s-r\right|} \) 
(since \( \left|s-r\right| \leq T \)), so it can actually be written as 
\( \leq C \sqrt{\left|s-r\right|} \), where the constant \( C \) does not depend on the control.\\
Thus from Lemma \ref{lemma continuity}, we get
\begin{align*}
	\rho_F\big(\mu^{t,m,\alpha}_s, \mu^{t,m,\alpha}_r\big)
	~\leq~
	C \sqrt{\left|s-r\right|},
\end{align*}
which proofs the continuity.
\end{proof}

\subsubsection{Continuity of the Hamiltonian}
	\begin{lemma}\label{ham_cts}
		Let Assumption \ref{A.1} hold true and \( G \) : \( [0, T] \times \Mc_2(\R^d) \times B^{d,\lambda+1}_l \times B^{d \times d,\lambda}_l \times B^{d,\lambda+2}_l \rightarrow\R \) be  defined by
\begin{align}
    G(t, m, p, q, r) = & \inf_{\tilde{\alpha}\in\tilde{\Ac}} \big[
    \left\langle L\left(t, \cdot, m, \tilde{\alpha}\right), m \right\rangle 
    + \langle b(t, \cdot, m, \tilde{\alpha})p(\cdot) \nonumber \\
    &+ \frac{1}{2}\Tr\left(\sigma\sigma^{\top}\big(t,\cdot,m,\tilde{\alpha}\big)q(\cdot)\right) \nonumber\\
    &+ \gamma\big(t,\cdot,m,\tilde{\alpha}\big) \sum_{\ell\geq0}(\ell - 1)p_\ell\big(t,\cdot,m,\tilde{\alpha}\big)r(\cdot),m \rangle \big].
\end{align}
Then G is continuous in each variable in the sense that for any \(\theta = (t,m)\), \((\theta,p,q,r) \in  [0, T] \times \Mc_2(\R^d) \times B^{d,\lambda+1}_l \times B^{d \times d,\lambda}_l \times B^{d,\lambda+2}_l \), then for any $\eps > 0$, there exists a $\delta>0$ such that for any \((\theta',p',q',r') \in  [0, T] \times \Mc_2(\R^d) \times B^{d,\lambda+1}_l \times B^{d \times d,\lambda}_l \times B^{d,\lambda+2}_l \) satisfying
		\[d_F^2(\theta,\theta') + |p-p'|_{l}+|q-q'|_{l}+|r-r'|_{l} \leq \delta, \]
		we have 
		\[
		|G(\theta,p,q,r) - G(\theta',p',q',r')| \leq \eps.
		\]
	\end{lemma}

	\begin{proof}
		First we have 
		\begin{align}\label{diff:G}
			&|G(t,m,p,q,r)-G(t',m',p',q',r')| \nonumber
			\\\leq &|G(t,m,p,q,r)-G(t,m,p',q',r')| + |G(t,m,p',q',r')-G(t',m',p',q',r')|.
		\end{align}
		Let \( H \) : \(A \times [0, T] \times \Mc_2(\R^d) \times B^{d,\lambda+1}_l \times B^{d \times d,\lambda}_l \times B^{d,\lambda+2}_l \rightarrow\R \) be the function defined by
		\begin{align}
			H(a,t, m, p, q, r) = & 
			\left\langle L\left(t, \cdot, m, a\right), m \right\rangle 
			+ \langle b(t, \cdot, m, a)p(\cdot) \nonumber \\
			&+ \frac{1}{2}\Tr\left(\sigma\sigma^{\top}\big(t,\cdot,m,a\big)q(\cdot)\right) \nonumber\\
			&+ \gamma\big(t,\cdot,m,a\big) \sum_{\ell\geq0}(\ell - 1)p_\ell\big(t,\cdot,m,a\big)r(\cdot),m \rangle.
		\end{align}
		\textbf{Part 1:} For the first term in \eqref{diff:G}, by definition we have the inequality
		\[
		|G(t,m,p_{1},q_{1},r_{1})-G(t,m,p_{2},q_{2},r_{2})|\leq \sup _{\tilde{\alpha}\in\tilde{\Ac}}|H(\tilde{\alpha},t,m,p_{1},q_{1},r_{1})-H(\tilde{\alpha},t,m,p_{2},q_{2},r_{2})|,
		\]
		and then it is suffice to show that for any $\tilde{\alpha}\in\tilde{\Ac}$,
		\begin{gather*}
		|H(\tilde{\alpha},t,m,p,q,r)-H(\tilde{\alpha},t,m,p',q',r')|  \\
		\leq L_{G}\left(\int_{\mathbb{R}^{d}} 1+|x|m(dx)\right)\times\left(|p-p'|_{l}+|q-q'|_{l}+|r-r'|_{l}\right),
		\end{gather*}
		where $L_{G}$ is some positive constant independent of $\tilde{\alpha}$. Indeed from the construction of $H$, it is straightforward that
		
		\begin{align*}
		&|H(\tilde{\alpha},t,m,p,q,r)-H(\tilde{\alpha},t,m,p',q',r')| \\
		&\leq C\left(\int_{\mathbb{R}^{d}}|p(x)-p'(x)|+|q(x)-q'(x)|+|r(x)-r'(x)|m(dx)\right) \\
		&\leq C\left(\left(\int_{\mathbb{R}^{d}}1+|x|m(dx)\right)\times\left(|p-p'|_{l}+|q-q'|_{l}\right)+|r-r'|_{l}\right),
		\end{align*}
		and hence verifies the continuity of G with respect to \((p,q,l)\).\\
		\textbf{Part 2:} For the second term in \eqref{diff:G}, let \( \Gamma(t,x,a) \) denotes \(\gamma\big(t,x,a\big)\sum_{\ell\geq0}(\ell - 1)p_\ell\big(t,x,a\big)\)
		we obtain
		\[
		\begin{aligned}
		&H\left(\tilde{\alpha}, t,m,p,q, r\right)-H\left(\tilde{\alpha},t', m' , p,q, r\right)\\
		&\leq \left\langle L\left(t, \cdot, m, \tilde{\alpha}(\cdot)\right), m \right\rangle-\left\langle L\left(t', \cdot, m', \tilde{\alpha}(\cdot,)\right), m' \right\rangle\\
		&\quad + \langle b(t, \cdot, m, \tilde{\alpha}(\cdot))p(\cdot), m\rangle-\langle b(t', \cdot, m', \tilde{\alpha}(\cdot))p(\cdot), m'\rangle \\
		&\quad + \langle \Gamma(t, \cdot, m, \tilde{\alpha}(\cdot))r(\cdot), m\rangle-\langle \Gamma(t', \cdot, m', \tilde{\alpha}(\cdot))r(\cdot), m'\rangle. \\
		&\quad + \langle \frac{1}{2}\Tr\left(\sigma\sigma^{\top}\big(t,\cdot,m,\tilde{\alpha}(\cdot)\big)q(\cdot)\right), m\rangle-\langle \frac{1}{2}\Tr\left(\sigma\sigma^{\top}\big(t',\cdot,m',\tilde{\alpha}(\cdot)\big)q(\cdot)\right), m'\rangle. 
		\end{aligned}
		\]
		Within Assumptions \ref{A.1} and \ref{A.3}.
		The first term of the right hand side can be bounded by 
		\begin{align*}
			&\langle L(t, \cdot, m, \tilde{\alpha}(\cdot)), m\rangle-\langle L(t', \cdot, m', \tilde{\alpha}(\cdot)), m'\rangle \\
			\leq& \langle L(t, \cdot, m, \tilde{\alpha}(\cdot)), m\rangle-\langle L(t', \cdot, m', \tilde{\alpha}(\cdot)), m\rangle 
			+ \langle L(t', \cdot, m', \tilde{\alpha}(\cdot)), m\rangle-\langle L(t', \cdot, m', \tilde{\alpha}(\cdot)), m'\rangle \\
			\leq& C_L d_F(\theta,\theta') m(\R^d)+ \int_{\mathbb{R}^{d}} L(t',x,m',\tilde{\alpha}(\cdot) ) (m-m')(dx). 
		\end{align*}
		The second term of the right hand is bounded by
		\begin{align*}
			&\langle b(t, \cdot, m, \tilde{\alpha}(\cdot))p(\cdot), m\rangle-\langle b(t', \cdot, m', \tilde{\alpha}(\cdot))p(\cdot), m'\rangle \\
			\leq& \langle b(t, \cdot, m, \tilde{\alpha}(\cdot))p(\cdot), m\rangle-\langle b(t', \cdot, m', \tilde{\alpha}(\cdot))p(\cdot), m\rangle \\
			&+ \langle b(t', \cdot, m', \tilde{\alpha}(\cdot))p(\cdot), m\rangle-\langle b(t', \cdot, m', \tilde{\alpha}(\cdot))p(\cdot), m'\rangle \\
			\leq& C_b d_F(\theta,\theta')\bigg(\int_{\mathbb{R}^{d}} 1+|x|m(dx)\bigg) |p|_{l} + \int_{\mathbb{R}^{d}} b(t',x,m',\tilde{\alpha}(\cdot))^{\top} p(x)(m-m')(dx) \\
			\leq& C \bigg(d_F(\theta,\theta')(\int_{\mathbb{R}^{d}} 1+|x|m(dx)) |p|_{l} + |b^{\top} p|_{\lambda} |m-m'|_{-\lambda}\bigg).
		\end{align*}
		The third term of the right hand is bounded by
		\begin{align*}
			&\langle \Gamma(t, \cdot, m, \tilde{\alpha}(\cdot))r(\cdot), m\rangle-\langle \Gamma(t', \cdot, m', \tilde{\alpha}(\cdot))r(\cdot), m'\rangle \\
			\leq& \langle \Gamma(t, \cdot, m, \tilde{\alpha}(\cdot))r(\cdot), m\rangle-\langle \Gamma(t', \cdot, m', \tilde{\alpha}(\cdot))r(\cdot), m\rangle \\
			&+ \langle \Gamma(t', \cdot, m', \tilde{\alpha}(\cdot))r(\cdot), m\rangle-\langle \Gamma(t', \cdot, m', \tilde{\alpha}(\cdot))r(\cdot), m'\rangle \\
			\leq& C_{\Gamma}d_F(\theta,\theta')\bigg(\int_{\mathbb{R}^{d}} 1+|x|m(dx)\bigg) |r|_{l} + \int_{\mathbb{R}^{d}} \Gamma(t',x,m',\tilde{\alpha}(\cdot))^{\top} r(x) (m-m')(dx) \\
			\leq& C \bigg(d_F(\theta,\theta')(\int_{\mathbb{R}^{d}} 1+|x|m(dx)) |r|_{l} + |\Gamma^{\top} r|_\lambda |m-m'|_{-\lambda}\bigg).
		\end{align*}
		The last term of the right hand is bounded by
		\begin{align*}
			&\langle \frac{1}{2}\Tr\left(\sigma\sigma^{\top}\big(t,\cdot,m,\tilde{\alpha}(\cdot)\big)q(\cdot)\right), m\rangle-\langle \frac{1}{2}\Tr\left(\sigma\sigma^{\top}\big(t',\cdot,m',\tilde{\alpha}(\cdot)\big)q(\cdot)\right), m'\rangle \\
			\leq& \langle \frac{1}{2}\Tr\left(\sigma\sigma^{\top}\big(t,\cdot,m,\tilde{\alpha}(\cdot)\big)q(\cdot)\right), m\rangle-\langle \frac{1}{2}\Tr\left(\sigma\sigma^{\top}\big(t',\cdot,m',\tilde{\alpha}(\cdot)\big)q(\cdot)\right), m\rangle \\
			&+ \langle \frac{1}{2}\Tr\left(\sigma\sigma^{\top}\big(t',\cdot,m',\tilde{\alpha}(\cdot)\big)q(\cdot)\right), m\rangle-\langle \frac{1}{2}\Tr\left(\sigma\sigma^{\top}\big(t',\cdot,m',\tilde{\alpha}(\cdot)\big)q(\cdot)\right), m'\rangle \\
			\leq& C_{\sigma}d_F(\theta,\theta')\bigg(\int_{\mathbb{R}^{d}} 1+|x|m(dx)\bigg) |q|_{l} + \int_{\mathbb{R}^{d}} \frac{1}{2}\Tr\left(\sigma\sigma^{\top}\big(t',\cdot,m',\tilde{\alpha}(\cdot)\big)q(\cdot)\right)(m-m')(dx) \\
			\leq& C \bigg(d_F(\theta,\theta')(\int_{\mathbb{R}^{d}} 1+|x|m(dx)) |q|_{l} + |\sigma\sigma^{\top}q|_\lambda |m-m'|_{-\lambda}\bigg).
		\end{align*}
		By Theorem \ref{soblev multiplication}, we have \(|b^\top p|_\lambda < \infty\), and the same inequality hold also for \(|\sigma\sigma^{\top}q|_\lambda\) and \(|\Gamma^{\top} r|_\lambda\). Therefore combining all the terms we get,
		\begin{align*}
			H\left(a, t,m,p,q, r\right)-H\left(a,t', m' , p,q, r\right) \leq & C d_F(\theta,\theta')(\int_{\mathbb{R}^{d}} 1+|x|m(dx))(1 + |p|_{l} + |q|_{l} + |r|_{l}) \\
			& + C' |m - m'|_{-\lambda}.
		\end{align*}
		According to Definition \ref{Bessel operator}, for $\eta = m_1 - m_2$, we have:
		\begin{align}
			|\eta|_{-\lambda}^2 = \int_{\mathbb{R}^d} (1 + |{k}|^2)^{-\lambda} |\mathcal{F} \eta({k})|^2 d{k},
		\end{align}
			which is exactly the $L_2$ norm of $|\mathcal{F}(m_1 - m_2)({k})|$ with weight $(1 + |{k}|^2)^{-\lambda}$, then we have
		\begin{align}
			F_k(m_1 - m_2) = \langle f_k, m_1 - m_2 \rangle = \int_{\mathbb{R}^d} f_k(x)(m_1 - m_2)(dx) = (2\pi)^{d/2} \mathcal{F}(m_1 - m_2)(k).
		\end{align}
		Therefore,
		\[
|F_{k}(m_1-m_2)| = (2\pi)^{d/2}|\mathcal{F}(m_1-m_2)(k)|.
\]
Plug into the definition of $\rho_F$, we finally get
\[
\begin{aligned}
\rho_{F}^{2}(m_1,m_2) &= \int_{\mathbb{R}^{d}} \frac{\left| F_{k}(m_1-m_2) \right|^{2}}{(1+|k|^{2})^{\lambda}} dk \\
&= \int_{\mathbb{R}^{d}} \frac{(2\pi)^{-d} |\mathcal{F}(m_1-m_2)(k)|^{2}}{(1+|k|^{2})^{\lambda}} dk \\
&= (2\pi)^{-d} \int_{\mathbb{R}^{d}} (1+|k|^{2})^{-\lambda} |\mathcal{F}(m_1-m_2)(k)|^{2} dk \\
&= (2\pi)^{-d} |m_1-m_2|_{-\lambda}^{2},
\end{aligned}
\]
		and hence verifies the continuity of G with respect to \((t,m)\).
		\end{proof}

\subsection{Synthesis of the Proof of Theorem \ref{hjbprop}}\label{section5.3}

Recall the definition of $G$ in Lemma \ref{ham_cts}, and $H$ in Lemma \ref{ham_cts} and \eqref{eq:HJB}.

\subsubsection{Commutator estimates}
For given functions $a$, $b$, define the operators
\[
\mathcal{B}f(x) = b(x)^{\top} D_{x} f(x), \quad 
\mathcal{A}f(x) = \frac{1}{2} \operatorname{Tr}\left(a(x) D_{x}^{2} f(x)\right).
\]

\begin{proposition}[{\cite[Proposition 5.6]{Erhan He 2025}}]\label{Commutator}
	Let $a \in H^{\lambda+d/2+1}(\mathbb{R}^{d})$ and $b, c \in H^{\lambda}(\mathbb{R}^{d})$. Suppose further that there exists a constant $\delta > 0$ such that $a$ satisfies the uniform ellipticity condition $\xi^{\top} a(x) \xi \geq \delta |\xi|^{2}$ for all $x, \xi \in \mathbb{R}^{d}$. Then, for any finite signed measure $\eta$, we have$$\int_{\mathbb{R}^{d}} \left(\mathcal{A}(\mathcal{J}_{2\lambda}\eta)(x) + \mathcal{B}(\mathcal{J}_{2\lambda}\eta)(x) + c(x)^{\top}(\mathcal{J}_{2\lambda}\eta)(x) \right) \eta(dx) \leq -\frac{\delta}{4} |\eta|_{1-\lambda}^{2} + C |\eta|_{-\lambda}^{2}, $$where $C > 0$ is a constant depending only on $\sup_{i,j} |a_{i,j}|_{\lambda+d/2+1}$, and $\mathcal{J}$ denotes the Bessel potential operator defined in Definition \ref{Bessel operator}.
\end{proposition}

\subsubsection{Continuity of value function}

\begin{lemma}
\label{prop:continuity_value_function}
Under Assumptions \ref{A.1} and \ref{A.3}, the value function $v(t, m)$ defined in \eqref{eq:value function} is continuous on $[0, T] \times \mathcal{M}_2(\mathbb{R}^d)$ under the product topology, where \( \mathcal{M}_2(\mathbb{R}^d) \) is equipped with the weak convergence topology.
\end{lemma}

\begin{proof}
	\textbf{Measure continuity:}
		It suffices to show that our value function $v(t,m)$ is continuous on $m$ under the $\mathbf{d}_{\mathrm{BL}}(m, m')$ metric.
	Let $m_n$ be a sequence of non-negative measures that converges to  $m$ in $\mathbf{d}_{\mathrm{BL}}(m, m')$, by the definition of the value function, for any $\eps > 0$, there exist controls $\alpha_{n} \in \mathcal{A}$ and $\alpha \in \mathcal{A}$ such that
	\[
	v(t, m_n) \geq J(t, m_n, \alpha_{n}) - \eps, \quad  v(t, m) \geq J(t, m, \alpha) - \eps.
	\]
	Also, by the infimum property,
	\[
	v(t, m_n) \leq J(t, m_n, \alpha_n) , \quad v(t, m) \leq J(t, m, \alpha). 
	\]
	Combining these inequalities, we get
	\[
	J(t, m_n, \alpha_n) - J(t, m, \alpha_n) - \eps \leq v(t, m_n) - v(t, m) \leq J(t, m_n, \alpha) - J(t, m, \alpha) + \eps. 
	\]
	Therefore, to prove the continuity of the value function at $(t,m)$, it suffices to show that, there is a constant $C > 0$, such that
	\[
	|J(t, m_n, \alpha) - J(t, m, \alpha)| \leq C \mathbf{d}_{\mathrm{BL}}(m, m_n),
		~\mbox{for all control}~ \alpha \in \mathcal{A}.
	\]
	Recall the cost functional:
	\[
	J(t, m, \alpha) = \int_t^T \left\langle L(s, \cdot, \mu_s^{t, m, \alpha}, \alpha_s(\cdot)) , \mu_s^{t, m, \alpha} \right\rangle ds + \left\langle g(\cdot, \mu_T^{t, m, \alpha}), \mu_T^{t, m, \alpha} \right\rangle.
	\]
	We will show that both terms converge as $m_n \to m$.
	Now, by Assumption \ref{A.3} of $L$ and $g$, we shall apply the triangle inequality to get for any $s \in [t, T]$,
	\begin{align*}
		&\left| \left\langle L(s, \cdot, \mu_s^{t, m_n, \alpha}, \alpha_s(\cdot)) , \mu_s^{t, m_n, \alpha} \right\rangle - \left\langle L(s, \cdot, \mu_s^{t, m, \alpha}, \alpha_s(\cdot)) , \mu_s^{t, m, \alpha} \right\rangle \right| \nonumber \\
		\leq &\left| \left\langle L(s, \cdot, \mu_s^{t, m_n, \alpha}, \alpha_s(\cdot)) , \mu_s^{t, m_n, \alpha} \right\rangle - \left\langle L(s, \cdot, \mu_s^{t, m, \alpha}, \alpha_s(\cdot)) , \mu_s^{t, m_n, \alpha} \right\rangle \right| \nonumber \\
		&+ \left| \left\langle L(s, \cdot, \mu_s^{t, m, \alpha}, \alpha_s(\cdot)) , \mu_s^{t, m_n, \alpha} \right\rangle - \left\langle L(s, \cdot, \mu_s^{t, m, \alpha}, \alpha_s(\cdot)) , \mu_s^{t, m, \alpha} \right\rangle \right| \nonumber \\
		\leq &C_{L,g} \langle \mu_s^{t, m, \alpha},1 \rangle \rho_F(\mu_s^{t, m, \alpha}, \mu_s^{t, m_n, \alpha}) 
		\quad \leq \quad C \mathbf{d}_{\mathrm{BL}}(\mu_s^{t, m, \alpha}, \mu_s^{t, m_n, \alpha}) \\
		\leq & C \mathbf{d}_{\mathrm{BL}}(m, m_n).
	\end{align*}
		And similar inequality also holds for the term \(\left\langle g(\cdot, \mu_T^{t, m, \alpha}), \mu_T^{t, m, \alpha} \right\rangle\), so we get the continuity of value function with respect to $m$.

		\textbf{Time continuity:}
		Fix $m \in \mathcal{M}_2(\mathbb{R}^d)$. Let $t_n \to t$, where $t_n, t \in [0,T]$. Without loss of generality, assume $t_n \ge t$.  
		By the dynamic programming principle \ref{DPP}, for any $n$, there exist a control $\alpha_{n}\in\Ac$ such that
		\begin{align*}
			v(t,m) + 1/n 
			~\geq~
				\int_{t}^{t_n} \<L\big(u,\cdot,\mu^{t,m,\alpha_n}_u, \alpha_n(\cdot)\big),\mu^{t,m,\alpha_n}_u\>du
			+
			v(t_n, \mu^{t,m,\alpha_n}_{t_n}).
		\end{align*}
		On the other hand, for any $\alpha\in\Ac$, we have
		\begin{equation}
			v(t,m) \leq \int_{t}^{t_{n}} \left\langle L\left(u,\cdot,\mu_{u}^{t,m,\alpha},\alpha\left(u,\cdot,\mu_{u}^{t,m,\alpha}\right)\right),\mu_{u}^{t,m,\alpha}\right\rangle du + v(t_{n},\mu_{t_{n}}^{t,m,\alpha}).
			\end{equation}
		Combine these two inequalities we get
		\[
\left|v(t, m)-v\left(t_{n},\mu_{t_{n}}^{t, m,\alpha_{n}}\right)\right|
\leq\int_{t}^{t_{n}}\left|\left\langle L,\mu_{u}^{t,m,\alpha_{n}}\right\rangle\right|du+\frac{1}{n}.
\]
From Lemma \ref{A.3} and Lemma \ref{lemma integrability K}, we know that $L$ and the mass of $\mu_{u}^{t,m,\alpha_{n}}$ are both bounded, therefore when n goes to infinity, 
\[
\left| v(t, m) - v\left(t_n, \mu_{t_n}^{t, m, \alpha_n}\right) \right| \rightarrow 0.
\]
Now consider $\left| v(t_{n},\mu_{t_{n}}^{t,m,\alpha_{n}}) - v(t_{n},m) \right|$.
By Lemma \ref{strong continuity}, we have $\rho_{F}(\mu_{t_{n}}^{t,m,\alpha_{n}},m)\leq C\sqrt{t_{n}-t}\to 0$.
Furthermore, due to the continuity of $v$ with respect to the measure as established above, it follows that the difference of $v$ at the points $(t_{n},\mu_{t_{n}}^{t,m,\alpha_{n}})$ and $(t_{n},m)$ tends to 0, i.e.:
\begin{equation*}
\left| v(t_{n},\mu_{t_{n}}^{t,m,\alpha_{n}}) - v(t_{n},m) \right| \to 0. 
\end{equation*}
Combining these two results, we finally get
\[
|v(t,m)-v(t_{n},m)|\leq|v(t,m)-v(t_{n},\mu_{t_{n}}^{t,m,\alpha_{n}})|+|v(t_{n},\mu_{t_{n}}^{t,m,\alpha_{n}})-v(t_{n},m)|\rightarrow 0.
\]
In conclusion, \( v \) is continuous on \( [0,T]\times\mathcal{M}_{2}(\mathbb{R}^{d}) \) endowed with the product topology.
	\end{proof}

\subsubsection{Existence of viscosity solution to HJB \eqref{eq:HJB}}

	\textbf{Part 1.} \( v \) is a viscosity subsolution:
	Let $\mu^{t,m,\alpha}_s,s\in (t,T]$ be the process as defined in \eqref{eq:mean measure xi} and \cite[Lemma 3.3]{Claisse Kang Lan Tan 2025} and
	\((b, p, q, r) \in {J}^{1,+}v(t_0, m_0)\). By definition there exists a \(C^{1,2}\) function \(\phi\) such that:
\begin{enumerate}
    \item \(v - \phi\) has a local maximum at \((t_0, m_0)\),
    \item \(\left( \partial_t \phi(t_0, m_0), \partial_{x}{\delta \phi}(t_0,m_0), \partial^2_{xx}{\delta} \phi(t_0,m_0),{\delta \phi}(t_0,m_0) \right) = (b, p, q, r)\).
\end{enumerate}
	Since By definition, we have
	\begin{align*}
		& G(t_0, m_0, \partial_{x}{\delta \phi}(t_0,m_0), \partial^2_{xx}{\delta} \phi(t_0,m_0),{\delta \phi}(t_0,m_0))  \\
		= & \inf_{\tilde{\alpha} \in \tilde{\Ac}} H(\tilde{\alpha},t_0, m_0, \partial_{x}{\delta \phi}(t_0,m_0), \partial^2_{xx}{\delta} \phi(t_0,m_0),{\delta \phi}(t_0,m_0)).
	\end{align*}
	For any \(\eps > 0\), there exists \(\tilde{\alpha}_\eps \in \tilde{\Ac} \) such that:
\[
H(\tilde{\alpha}_\eps) \leq G + \eps.
\]
Then define a control function \(\alpha_0(s,x) = \tilde{\alpha}_\eps(x)\) on \([t,T] \times \R^d\). Clearly \(\alpha_0 \in \mathcal{A}\). \\
	Then from the DPP Theorem \ref{DPP} applied to $(t_0,m_0)$, we have
	\[
		v(t_0,m_0) \leq
			\int_{t_0}^{t_0+h} \<L\big(u,\cdot,\mu^{t_0,m_0,\alpha_0}_u, \tilde{\alpha}_\eps(\cdot)\big),\mu^{t_0,m_0,\alpha_0}_u\>du
		+
		v(t_0+h, \mu^{t_0,m_0,\alpha_0}_{t_0+h}).
	\]
	WLOG we may assume that $v(t_0,m_0) = \phi(t_0,m_0)$, and then we have
	\begin{align*}
		\phi(t_0,m_0) &\leq
			\int_{t_0}^{t_0+h} \<L\big(u,\cdot,\mu^{t_0,m_0,\alpha_0}_u, \tilde{\alpha}_\eps(\cdot)\big),\mu^{t_0,m_0,\alpha_0}_u\>du
		+
		\phi(t_0+h, \mu^{t_0,m_0,\alpha_0}_{t_0+h}),\\
		0 &\leq
		\frac{1}{h}\int_{t_0}^{t_0+h} \<L\big(u,\cdot,\mu^{t_0,m_0,\alpha_0}_u, \tilde{\alpha}_\eps(\cdot)\big),\mu^{t_0,m_0,\alpha_0}_u\>du
		+
		\frac{\phi(t_0+h, \mu^{t_0,m_0,\alpha_0}_{t_0+h}) - \phi(t_0,m_0)}{h}.
	\end{align*}
	Applying Itô's formula \eqref{proposition Ito formula} to $\phi(t,m)$ between $t_0$ and $t_0+h$, we get
	\begin{align*}
		\phi(t_0+h, \mu^{t_0,m_0,\alpha_0}_{t_0+h}) - \phi(t_0,m_0)
		&~=~
		\int_{t_0}^{t_0+h}\partial_t \phi(u, \mu^{t_0,m_0,\alpha_0}_u)du \\
		~+&
		\int_{t_0}^{t_0+h}
		\big\<
			b\big(u,\cdot,\mu^{t_0,m_0,\alpha_0}_u,\tilde{\alpha}_\eps(\cdot)\big)\cdot  \partial_{x}{\delta \phi}(u,\mu^{t_0,m_0,\alpha_0}_u) (\cdot)\\
			~+&
			\frac{1}{2}\Tr\big(\sigma\sigma^{\top} \big(u, \cdot,\mu^{t_0,m_0,\alpha_0}_u,\tilde{\alpha}_\eps(\cdot)\big)\partial^2_{xx}{\delta} \phi(u,\mu^{t_0,m_0,\alpha_0}_u) (\cdot)\big)\\
			~+&
			\gamma\big(\cdot\big)\sum_{\ell\geq0}(\ell - 1)p_\ell\big(u,\cdot,\mu^{t_0,m_0,\alpha_0}_u,\tilde{\alpha}_\eps(\cdot)\big){\delta \phi}(\mu^{t_0,m_0.\alpha_0}_u)(\cdot)
			,
			\mu^{t_0,m_0,\alpha_0}
		\big\>
		du.
	\end{align*}
	And thus we have
	\begin{align*}
		0 &\leq
		\frac{1}{h}\int_{t_0}^{t_0+h}  
		\partial_t \phi(u, \mu^{t_0,m_0,\alpha_0}_u) + 
		\<L\big(u,\cdot,\mu^{t_0,m_0,\alpha_0}_u, \tilde{\alpha}_\eps(\cdot)\big) + \Gc^{\alpha_{0}}_u \phi(u,\mu^{t_0,m_0,\alpha_0})(\cdot)
		,\mu^{t_0,m_0,\alpha_0}_u\>du.
	\end{align*}
	From the previous lemma and coefficient continuity we know:
\begin{itemize}
    \item \( \mu^{t_0,m_0,\alpha_0} : [t_0,T] \to \Mc_2(\R^d) \) is continuous;
    \item Similar continuity holds for \( (L, \Gc^{\alpha_0} \phi_n)\) as well as all the derivatives of $\phi$;
    \item \(\alpha_0\) is constant.
\end{itemize}
	Thus we shall send $h$ to 0, and get
	\begin{align*}
		0 &\leq 
		\partial_t \phi(t_0, m_0) + 
		\<L\big(t_0,\cdot,m_0, \tilde{\alpha}_\eps(\cdot)\big) + \Gc^{\alpha_0}_{t_0} \phi(t_0,m_0)(\cdot)
		,m_0\>.
	\end{align*}
	Note that 
	\begin{align*}
		\<L\big(t_0,\cdot,m_0, \tilde{\alpha}_\eps(\cdot)\big) + \Gc^{{\alpha}_0}_{t_0} \phi(t_0,m_0)(\cdot)
		,m_0\> = H(\tilde{\alpha}_\eps(\cdot),t_0,m_0) \leq G(t_0,m_0) + \eps.
	\end{align*}
	Since $\eps$ is arbitrary, this shows that
	\[-\frac{\partial \phi}{\partial t}(t_0, m_0) 
		- G(t_0, m_0, \partial_{x}{\delta \phi}(t_0,m_0), \partial^2_{xx}{\delta} \phi(t_0,m_0),{\delta \phi}(t_0,m_0)) 
		\leq 0.\]
	\textbf{Part 2.}\( v \) is a viscosity supersolution:\\ 
    Let \((b, p, q, r) \in {J}^{1,-}v(\theta_0)\), where \(\theta_0 =  (t_0, m_0) \in [t,T] \times \Mc_2(\R^d)\), there exists a \(C^{1,2}\) function \(\phi\) such that:
\begin{enumerate}
    \item \(v - \phi\) has a local minimum at \((t_0, m_0)\),
    \item Derivatives equal to \((b, p, q, r)\).
\end{enumerate}
Suppose, contrary to the claim, that:
\[
-b - G(t_0, m_0, p, q, r) < 0.
\]
Then there exists \(\delta > 0\) such that:
\[
b + G(t_0, m_0, p, q, r) > 2\delta > 0.
\]
For \(\Delta \in \R^+\), we shall define a ball on \([t,T] \times \Mc_2(\R^d)\) by
\begin{align*}
	B_{\Delta}(\theta_0) := \left\{
		\nu = (t_\nu,m_\nu) \in [t,T] \times \Mc_2(\R^d) :
		d_F(\nu,\theta_0) = |t_0 - t_\nu| \times \rho_F\big(m_0, m_\nu\big) < \Delta
		\right\}.
\end{align*}
 By continuity of all derivatives of \(\phi\) and Lemma \ref{ham_cts} there exist a \(\Delta \in \R^+\) such that for any \(\theta \in B_{\Delta}(\theta_0)\),
 \[
\partial_t \phi(\theta) + G\left(\theta, \partial_{x}{\delta \phi}(\theta), \partial^2_{xx}{\delta} \phi(\theta),{\delta \phi}(\theta)\right) > {\delta} > 0.
\]
And by Lemma \ref{strong continuity}, there exists a $h$ such that for any \(\alpha \in \Ac\), and any \(s \in [t_0,t_0+h]\),
    \[
\partial_t \phi(s, \mu^{t_0,m_0,\alpha}_s) + G\left(s, \mu^{t_0,m_0,\alpha}_s, \partial_{x}{\delta} \phi(s, \mu^{t_0,m_0,\alpha}_s), \partial^2_{xx}{\delta} \phi(s, \mu^{t_0,m_0,\alpha}_s), {\delta \phi}(s, \mu^{t_0,m_0,\alpha}_s)\right) > {\delta}.
\]
    From the DPP Theorem \ref{DPP}, let \(\eps = \frac{\delta h}{2}\) there exists a $\eps$-control $\alpha^{\eps,h}\in\Ac$ such that
	\begin{align*}
		v(t_0,m_0) + \eps 
		~\geq~
			\int_{t_0}^{t_0+h} \<L\big(u,\cdot,\mu^{t_0,m_0,\alpha^{\eps,h}}_u, \alpha^{\eps,h}(\cdot)\big),\mu^{t_0,m_0,\alpha^{\eps,h}}_u\>du
		+
		v(t_0+h, \mu^{t_0,m_0,\alpha^{\eps,h}}_{t_0+h}),
	\end{align*}
    and again we assume $v(t_0,m_0) = \phi(t_0,m_0)$, then we have
	\begin{align*}
		\eps ~\geq~
		\int_{t_0}^{t_0+h} \<L\big(u,\cdot,\mu^{t_0,m_0,\alpha^{\eps,h}}_u, \alpha^{\eps,h}(\cdot)\big),\mu^{t_0,m_0,\alpha^{\eps,h}}_u\>du
		+
		\phi(t_0+h, \mu^{t_0,m_0,\alpha^{\eps,h}}_{t_0+h}) - \phi(t_0,m_0).
	\end{align*}
	Applying Itô's formula \eqref{proposition Ito formula} to $\phi(t,m)$ between $t_0$ and $t_0+h$, we get
	\begin{align*}
	\eps ~&\geq~
		\int_{t_0}^{t_0+h}  
		\partial_t \phi(u, \mu^{t_0,m_0,\alpha^{\eps,h}}_u) + 
		\<L\big(u,\cdot,\mu^{t_0,m_0,\alpha^{\eps,h}}_u, \alpha^{\eps,h}(\cdot)\big) + \Gc^\alpha_u \phi(u,\mu^{t_0,m_0,\alpha^{\eps,h}})(\cdot)
		,\mu^{t_0,m_0,\alpha^{\eps,h}}_u\>du \\
        ~&\geq~
        \int_{t_0}^{t_0+h}  
		\partial_t \phi(u, \mu^{t_0,m_0,\alpha^{\eps,h}}_u) + 
		G\left(u, \mu^{t_0,m_0,\alpha^{\eps,h}}_u, \partial_{x}{\delta} \phi(u, \mu^{t_0,m_0,\alpha^{\eps,h}}_u), \partial^2_{xx}{\delta} \phi(u, \mu^{t_0,m_0,\alpha^{\eps,h}}_u), {\delta \phi}(u, \mu^{t_0,m_0,\alpha^{\eps,h}}_u)\right)du \\
        ~&\geq~
        \delta h.
	\end{align*}
	We thus get the desired contradiction.

\subsubsection{Uniqueness of viscosity solution to HJB \eqref{eq:HJB}}
For uniqueness, we need to verify that the Hamiltonian \( G \) : \( [0, T] \times \Mc_2(\R^d) \times B^{d,\lambda+1}_l \times B^{d \times d,\lambda}_l \times B^{d,\lambda+2}_l \rightarrow\R \) defined in \ref{ham_cts} satisfies Assumption \ref{A.ham} and hence satisfies the comparison result.\\
For the verification of Assumption \ref{A.ham}(i), it follows the exact same argument as the proof of Lemma \ref{ham_cts}, and thus we just omitted the repetitive details here.\\
As for Assumption \ref{A.ham}(ii), we shall rewrite the inequality in Lemma \ref{ham_cts} as
		\begin{align*}
			&H\left(\tilde{\alpha}, t_1,m_1,\nabla\kappa,\nabla^{2}\kappa, \kappa\right)-H\left(\tilde{\alpha},t_2, m_2 , \nabla\kappa,\nabla^{2}\kappa, \kappa\right)\nonumber \\
			&\leq \left\langle L\left(t_1, \cdot, m_1, \tilde{\alpha}(\cdot)\right), m_1 \right\rangle-\left\langle L\left(t_2, \cdot, m_2, \tilde{\alpha}(\cdot)\right), m_2 \right\rangle\\
			&\quad + \langle b(t_1, \cdot, m_1, \tilde{\alpha}(\cdot))\nabla\kappa(\cdot), m_1\rangle-\langle b(t_2, \cdot, m_2, \tilde{\alpha}(\cdot))\nabla\kappa(\cdot), m_2\rangle \\
			&\quad + \langle \Gamma(t_1, \cdot, m_1, \tilde{\alpha}(\cdot))\kappa(\cdot), m_1\rangle-\langle \Gamma(t_2, \cdot, m_2, \tilde{\alpha}(\cdot))\kappa(\cdot), m_2\rangle. \\
			&\quad + \langle \frac{1}{2}\Tr\left(\sigma\sigma^{\top}\big(t_1,\cdot,m_1,\tilde{\alpha}(\cdot)\big)\nabla^{2}\kappa(\cdot)\right), m_1\rangle-\langle \frac{1}{2}\Tr\left(\sigma\sigma^{\top}\big(t_2,\cdot,m_2,\tilde{\alpha}(\cdot)\big)\nabla^{2}\kappa(\cdot)\right), m_2\rangle.
		\end{align*}
		For the first term on the right hand side, applying dual inequality to get
		\begin{align*}
			&\langle L(t_1, \cdot, m_1, \tilde{\alpha}(\cdot)), m_1\rangle-\langle L(t_2, \cdot, m_2, \tilde{\alpha}(\cdot)), m_2\rangle \\
			\leq& \langle L(t_1, \cdot, m_1, \tilde{\alpha}(\cdot)), m_1\rangle-\langle L(t_2, \cdot, m_2, \tilde{\alpha}(\cdot)), m_1\rangle 
			+ \langle L(t_2, \cdot, m_2, \tilde{\alpha}(\cdot)), m_1\rangle-\langle L(t_2, \cdot, m_2, \tilde{\alpha}(\cdot)), m_2\rangle \\
			\leq& m_1(\R^d)C_L d_F(\theta_1,\theta_2) + \int_{\mathbb{R}^{d}} L(t_2,x,m_2,\tilde{\alpha}(\cdot) ) (m_1-m_2)(dx) \\
			\leq& m_1(\R^d)C_L d_F(\theta_1,\theta_2) + |L|_{\lambda} |m_1 - m_2|_{-\lambda}.
		\end{align*}
		The second term of the right hand side is bounded by
		\begin{align*}
			&\langle b(t_1, \cdot, m_1, \tilde{\alpha}(\cdot))\nabla\kappa(\cdot), m_1\rangle-\langle b(t_2, \cdot, m_2, \tilde{\alpha}(\cdot))\nabla\kappa(\cdot), m_2\rangle \\
			\leq& \langle b(t_1, \cdot, m_1, \tilde{\alpha}(\cdot))\nabla\kappa(\cdot), m_1\rangle-\langle b(t_2, \cdot, m_2, \tilde{\alpha}(\cdot))\nabla\kappa(\cdot), m_1\rangle \\
			&+ \langle b(t_2, \cdot, m_2, \tilde{\alpha}(\cdot))\nabla\kappa(\cdot), m_1\rangle-\langle b(t_2, \cdot, m_2, \tilde{\alpha}(\cdot))\nabla\kappa(\cdot), m_2\rangle \\
			\leq& m_1(\R^d) C_b d_F(\theta_1,\theta_2)|\nabla\kappa|_{L^{\infty}} + \int_{\mathbb{R}^{d}} b(t_2,x,m_2,\tilde{\alpha}(\cdot))^{\top} \nabla\kappa(x)(m_1-m_2)(dx).
		\end{align*}
		The third term of the right hand side is bounded by
		\begin{align*}
			&\langle \Gamma(t_1, \cdot, m_1, \tilde{\alpha}(\cdot))\kappa(\cdot), m_1\rangle-\langle \Gamma(t_2, \cdot, m_2, \tilde{\alpha}(\cdot))\kappa(\cdot), m_2\rangle \\
			\leq& \langle \Gamma(t_1, \cdot, m_1, \tilde{\alpha}(\cdot))\kappa(\cdot), m_1\rangle-\langle \Gamma(t_2, \cdot, m_2, \tilde{\alpha}(\cdot))\kappa(\cdot), m_1\rangle \\
			&+ \langle \Gamma(t_2, \cdot, m_2, \tilde{\alpha}(\cdot))\kappa(\cdot), m_1\rangle-\langle \Gamma(t_2, \cdot, m_2, \tilde{\alpha}(\cdot))\kappa(\cdot), m_2\rangle \\
			\leq& m_1(\R^d) C_{\Gamma}d_F(\theta_1,\theta_2)|\kappa|_{L^{\infty}} + \int_{\mathbb{R}^{d}} \Gamma(t_2,x,m_2,\tilde{\alpha}(\cdot))^{\top} \kappa(x) (m_1-m_2)(dx).
		\end{align*}
		The last term of the right hand side is bounded by
		\begin{align*}
			&\langle \frac{1}{2}\Tr\left(\sigma\sigma^{\top}\big(t_1,\cdot,m_1,\tilde{\alpha}(\cdot)\big)\nabla^{2}\kappa(\cdot)\right), m_1\rangle-\langle \frac{1}{2}\Tr\left(\sigma\sigma^{\top}\big(t_2,\cdot,m_2,\tilde{\alpha}(\cdot) \big)\nabla^{2}\kappa(\cdot)\right), m_2\rangle \\
			\leq& \langle \frac{1}{2}\Tr\left(\sigma\sigma^{\top}\big(t_1,\cdot,m_1,\tilde{\alpha}(\cdot)\big)\nabla^{2}\kappa(\cdot)\right), m_1\rangle-\langle \frac{1}{2}\Tr\left(\sigma\sigma^{\top}\big(t_2,\cdot,m_2,\tilde{\alpha}(\cdot)\big)\nabla^{2}\kappa(\cdot)\right), m_1\rangle \\
			&+ \langle \frac{1}{2}\Tr\left(\sigma\sigma^{\top}\big(t_2,\cdot,m_2,\tilde{\alpha}(\cdot)\big)\nabla^{2}\kappa(\cdot)\right), m_1\rangle-\langle \frac{1}{2}\Tr\left(\sigma\sigma^{\top}\big(t_2,\cdot,m_2,\tilde{\alpha}(\cdot)\big)\nabla^{2}\kappa(\cdot)\right), m_2\rangle \\
			\leq& m_1(\R^d) C_{\sigma}d_F(\theta_1,\theta_2)|\nabla^{2}\kappa|_{L^{\infty}} + \int_{\mathbb{R}^{d}} \frac{1}{2}\Tr\left(\sigma\sigma^{\top}\big(t_2,\cdot,m_2,\tilde{\alpha}(\cdot)\big)\nabla^{2}\kappa(\cdot)\right)(m_1-m_2)(dx).
		\end{align*}
		Combining all the terms we get
		\begin{align*}
			H&\left(\tilde{\alpha}, t_1,m_1,\nabla\kappa,\nabla^{2}\kappa, \kappa\right)-H\left(\tilde{\alpha},t_2, m_2 , \nabla\kappa,\nabla^{2}\kappa, \kappa\right)\nonumber \\
			\leq & m_1(\R^d) C_1 d_F(\theta_1,\theta_2) + m_1(\R^d) C_2 d_F(\theta_1,\theta_2)(|\nabla\kappa|_{L^{\infty}}+|\nabla^{2}\kappa|_{L^{\infty}}+|\kappa|_{L^{\infty}}) + C_3 |m_1 - m_2|_{-\lambda} \\
			&+  \int_{\mathbb{R}^{d}} b(t_2,x,m_2,\tilde{\alpha}(x))^{\top} \nabla\kappa(x)(m_1-m_2)(dx) 
			+  \int_{\mathbb{R}^{d}} \Gamma(t_2,x,m_2,\tilde{\alpha}(x))^{\top} \kappa(x) (m_1-m_2)(dx) \\
			&+  \int_{\mathbb{R}^{d}} \frac{1}{2}\Tr\left(\sigma\sigma^{\top}\big(t_2,\cdot,m_2,\tilde{\alpha}(x)\big)\nabla^{2}\kappa(\cdot)\right)(m_1-m_2)(dx).
		\end{align*}
From Assumption \ref{A.ham}, and Definition \ref{Bessel operator} we know that
\[
\kappa(x)=\frac{1}{\epsilon}\mathcal{F}^{-1}\left(\frac{\mathcal{F}(m_1-m_2)(k)}{\left(1+|k|^{2}\right)^{\lambda}}\right)(x) = \frac{1}{\epsilon}(\mathcal{J}_{2\lambda}(m_1-m_2))(x).
\]
Therefore, applying Proposition \ref{Commutator} with \(\eta = m_1 - m_2\), we get
\begin{align}\label{H_ineq2}
	&H\left(\tilde{\alpha}, t_1,m_1,\nabla\kappa,\nabla^{2}\kappa, \kappa\right)-H\left(\tilde{\alpha},t_2, m_2 , \nabla\kappa,\nabla^{2}\kappa, \kappa\right) \nonumber \\
	\leq & m_1(\R^d) C_1 d_F(\theta_1,\theta_2) + m_1(\R^d) C_2 d_F(\theta_1,\theta_2)(|\nabla\kappa|_{L^{\infty}}+|\nabla^{2}\kappa|_{L^{\infty}}+|\kappa|_{L^{\infty}}) \nonumber \\
	&+ C_3 |m_1 - m_2|_{-\lambda} - \frac{\delta}{4\eps}|m_1 - m_2|^2_{1-\lambda} + \frac{C_4}{\eps}|m_1 - m_2|^2_{-\lambda}.
\end{align}
By direct computation, the estimation for \( |\kappa(x)| \) is derived as follows:

\[
|\kappa(x)| = \left| \frac{1}{\eps} \int_{\mathbb{R}^d} \frac{\operatorname{Re}\left(F_k(m_1-m_2) f_k^*(x)\right)}{\left(1+|k|^2\right)^\lambda} dk \right| \leq \frac{1}{\eps} \int_{\mathbb{R}^d} \frac{\left|F_k(m_1-m_2)\right|\left|f_k^*(x)\right|}{\left(1+|k|^2\right)^\lambda} dk.
\]
Since \( |f_k^*(x)| = (2\pi)^{-d/2} \), applying Cauchy-Schwarz:

\[
\int_{\mathbb{R}^d} \frac{\left|F_k(m_1-m_2)\right|}{\left(1+|k|^2\right)^\lambda} dk \leq \left( \int_{\mathbb{R}^d} \frac{\left|F_k(m_1-m_2)\right|^2}{\left(1+|k|^2\right)^\lambda} dk \right)^{1/2} \left( \int_{\mathbb{R}^d} \frac{1}{\left(1+|k|^2\right)^\lambda} dk \right)^{1/2} = \rho_F(m_1,m_2) \cdot C,
\]
where \( C = \left( \int_{\mathbb{R}^d} \frac{1}{\left(1+|k|^2\right)^\lambda} dk \right)^{1/2} \) is a finite constant. Therefore:

\[
|\kappa(x)| \leq \frac{(2\pi)^{-d/2} C}{\epsilon} \rho_F(\mu,\nu) = \frac{C^{\prime}}{\epsilon} \rho_F(\mu,\nu),
\]
and for $|\nabla\kappa|$ we have
\begin{align*}
	\nabla\kappa(x) &= \nabla_{x}\left[\frac{1}{\eps}\int_{\mathbb{R}^{d}}\frac{\operatorname{Re}\left(F_{k}(m_1-m_2)f_{k}^{*}(x)\right)}{\left(1+|k|^{2}\right)^{\lambda}}dk\right] \\
	&= \frac{1}{\eps}\int_{\mathbb{R}^{d}}\frac{\operatorname{Re}\left(F_{k}(m_1-m_2)\nabla_{x}f_{k}^{*}(x)\right)}{\left(1+|k|^{2}\right)^{\lambda}}dk,
	\end{align*}
	where \(\nabla_{x}f_{k}^{*}(x) = -ikf_{k}^{*}(x)\), so
	\begin{align*}
	\nabla\kappa(x) &= \frac{1}{\eps}\int_{\mathbb{R}^{d}}\frac{\operatorname{Re}\left(-i kF_{k}(m_1-m_2)f_{k}^{*}(x)\right)}{\left(1+|k|^{2}\right)^{\lambda}}dk.
	\end{align*}
	Taking absolute value and using \( |\operatorname{Re}(z)| \leq |z|\) 
	\begin{align*}
	|\nabla\kappa(x)| &\leq \frac{1}{\eps}\int_{\mathbb{R}^{d}}\frac{|k|\left|F_{k}(m_1-m_2)\right|\left|f_{k}^{*}(x)\right|}{\left(1+|k|^{2}\right)^{\lambda}}dk.
	\end{align*}
	Since  \(|f_{k}^{*}(x)| = (2\pi)^{-d/2}\) is constant, apply Cauchy-Schwarz inequality:
	\begin{align*}
	\int_{\mathbb{R}^{d}}\frac{|k|\left|F_{k}(m_1-m_2)\right|}{\left(1+|k|^{2}\right)^{\lambda}}dk &\leq \left(\int_{\mathbb{R}^{d}}\frac{\left|F_{k}(m_1-m_2)\right|^{2}}{\left(1+|k|^{2}\right)^{\lambda}}dk\right)^{1/2}\left(\int_{\mathbb{R}^{d}}\frac{|k|^{2}}{\left(1+|k|^{2}\right)^{\lambda}}dk\right)^{1/2} \\
	&= \rho_{F}(m_1,m_2) \cdot C'',
	\end{align*}
	and thus \(|\nabla\kappa(x)| \leq \frac{C'' \rho_{F}(m_1,m_2)}{\eps}\).\\
	And finially, we have a similar estimation for $\left|\nabla^{2}\kappa(x)\right|$:
	\[
	\left|\nabla^{2}\kappa(x)\right|\leq\frac{1}{\epsilon}\int_{\mathbb{R}^{d}}\frac{|k|^{2}|F_{k}(\mu-\nu)|\left|f_{k}^{*}(x)\right|}{\left(1+|k|^{2}\right)^{\lambda}} dk,
	\]
	so \(|\nabla^2\kappa(x)| \leq \frac{C''' \rho_{F}(m_1,m_2)}{\eps}\). 

	Plag the inequalities into \eqref{H_ineq2}, we get 
\begin{align}\label{H_ineq3}
	&H\left(\tilde{\alpha}, t_1,m_1,\nabla\kappa,\nabla^{2}\kappa, \kappa\right)-H\left(\tilde{\alpha},t_2, m_2 , \nabla\kappa,\nabla^{2}\kappa, \kappa\right) \nonumber \\
	\leq & C m_1(\R^d) (d_F(\theta_1,\theta_2) +  \frac{1}{\eps}d_{F}^2(\theta_1,\theta_2) ) + C_3 |m_1 - m_2|_{-\lambda} - \frac{\delta}{4\eps}|m_1 - m_2|^2_{1-\lambda} + \frac{C_4}{\eps}|m_1 - m_2|^2_{-\lambda}.
\end{align}
{Since we have \(\rho_{F}^{2}(m_1,m_2) = (2\pi)^{-d} |m_1-m_2|_{-\lambda}^{2}\), we finally get the bound 
	\begin{align*}
		&H\left(\tilde{\alpha}, t_1,m_1,\nabla\kappa,\nabla^{2}\kappa, \kappa\right)-H\left(\tilde{\alpha},t_2, m_2 , \nabla\kappa,\nabla^{2}\kappa, \kappa\right) \nonumber \\
	\leq & C (m_1(\R^d)+1) (d_F(\theta_1,\theta_2) +  \frac{1}{\eps}d_{F}^2(\theta_1,\theta_2)).
	\end{align*}
	for some constant C, and this verifies Assumption \ref{A.ham}(ii).}

\end{document}